\documentclass[reqno,10pt]{amsart}

\usepackage{amsmath}
\usepackage{amssymb}
\usepackage{amsthm}
\usepackage{xspace}
\usepackage{pictexwd}
\usepackage{float}  \restylefloat{figure} 
\usepackage[OT2,OT1]{fontenc}
\usepackage{amscd}
\usepackage{enumerate}

\theoremstyle{plain}
\newtheorem{thm}{Theorem}[section]
\newtheorem*{thm*}{Theorem}
\newtheorem{prop}[thm]{Proposition}
\newtheorem{lem}[thm]{Lemma}

\theoremstyle{definition}
\newtheorem{mydef}{Definition}[section]

\theoremstyle{remark}
\newtheorem*{rem}{Remark}

\def\forcehmode{\hskip0pt\relax}

\let\myskip=\medskip

\def\definebb#1=#2.{\def#1{{{\mathbb #2}^{\vphantom{x}}}}}
\definebb \c=C.
\definebb \r=R.
\definebb \s=S.
\definebb \z=Z.
\definebb \h=H.

\def\Cal#1{\mathcal{#1}}

\def\calB{{\Cal B}}
\def\calV{{\Cal V}}

\def\hyp{\h}
\def\arf{\si}
\def\hsi{{\hat\si}}
\def\Gm{{G_m}}
\def\Pm{{G_m^+}}
\def\Nm{{G_m^-}}

\def\lev{{s}}
\def\levm{\lev_m}
\def\lat{\Ga}
\def\piorb{\pi_1}
\def\piorbO{\pi_1^0}

\def\hGa{\hat{\Ga}}

\def\dd{\partial}

\def\al{{\alpha}}
\def\be{{\beta}}
\def\Ga{{\Gamma}}
\def\ga{{\gamma}}
\def\de{{\delta}}
\def\De{{\Delta}}
\def\ve{{\varepsilon}}
\def\la{{\lambda}}
\def\si{{\sigma}}
\def\Si{{\Sigma}}
\def\om{{\omega}}

\let\emptyset=\varnothing
\def\st{\,\,\big|\,\,}
\def\<{\langle}
\def\>{\rangle}
\def\ie{i.e.\xspace}

\let\ge=\geqslant
\let\le=\leqslant

\DeclareMathOperator{\Aut}{Aut}

\DeclareMathOperator{\id}{id}
\DeclareMathOperator{\Hol}{Hol}

\let\mod=\undefined \DeclareMathOperator{\mod}{mod}

\begin{document}

\author[Sergey Natanzon]{Sergey Natanzon}
\address{National Research University Higher School of Economics, Vavilova Street 7, 117312 Moscow, Russia}
\address{Institute of Theoretical and Experimental Physics (ITEP), Moscow, Russia}
\email{natanzon@mccme.ru}
\author[Anna Pratoussevitch]{Anna Pratoussevitch}
\address{Department of Mathematical Sciences\\ University of Liverpool\\ Liverpool L69~7ZL}
\email{annap@liv.ac.uk}

\title{Higher Spin Klein Surfaces}

\begin{date}  {\today} \end{date}

\thanks{Grant support for S.N.: The article was prepared within the framework of the Academic Fund Program
at the National Research University Higher School of Economics (HSE) in 2015--16 (grant Nr 15-01-0052)
and supported within the framework of a subsidy granted to the HSE by the Government of the Russian Federation for the implementation of the Global Competitiveness Program.
Grant support for A.P.: The work was supported in part by the Leverhulme Trust grant RPG-057.}


\begin{abstract}
A Klein surface is a generalisation of a Riemann surface
to the case of non-orientable surfaces or surfaces with boundary.
The category of Klein surfaces is isomorphic to the category of real algebraic curves.
An $m$-spin structure on a Klein surface is a complex line bundle whose $m$-th tensor power is the cotangent bundle.
We describe all $m$-spin structures on Klein surfaces of genus greater than one
and determine the conditions for their existence. 
In particular we compute the number of $m$-spin structures on a Klein surface
in terms of its natural topological invariants.
\end{abstract}










\maketitle


\section{Introduction}

\myskip
Under an $m$-{\it spin Riemann surface} for an integer $m>1$
we understand a compact Riemann surface~$P$ with a complex line bundle $e:L\to P$
such that its $m$-th tensor power $e^{\otimes m}:L^{\otimes m}\to P$ is isomorphic to the cotangent bundle of~$P$
(compare with~\cite{Jarvis:2000}).
This is a natural generalisation of the classical ($m=2$) algebraic curves with Theta characteristics
studied by Riemann~\cite{R}. 
The moduli spaces of $m$-spin Riemann surfaces have been studied
because of their connections with integrable systems~\cite{Witten:1993}, \cite{FShZ}

\myskip
The invariants of an $m$-spin Riemann surface $(P,e)$
are given by the genus $g$ of~$P$ and the Arf invariant $\de\in\{0,1\}$.
The Arf invariant is determined by the parity of the dimension of the space of sections of the $m$-spin bundle,
see~\cite{A}, \cite{Mu}.
For a given Riemann surface of genus~$g$, the number of corresponding $m$-spin Riemann surfaces is $m^{2g}$.
For odd~$m$ the Arf invariant is always  $\de=0$.
For even~$m$, the number of $m$-spin Riemann surfaces with $\de=1$ and $\de=0$
is $2^{-1-g}m^{2g}(2^g-1)$ and $2^{-1-g}m^{2g}(2^g+1)$ respectively, see~\cite{Jarvis:2000}, \cite{NP:2005}, \cite{NP:2009}.

\myskip
{\it A Klein surface} is a generalisation of a Riemann surface
in the case of non-orientable surfaces or surfaces with boundary.
A Klein surface is a quotient $P/\tau$,
where $P$ is a compact Riemann surface and $\tau:P\to P$ is an anti-holomorphic involution on~$P$.
The category of such pairs $(P,\tau)$ is isomorphic to the category of real algebraic curves,
see~\cite{AllingGreenleaf:1971}.

\myskip
The boundary of the surface $P/\tau$ corresponds to the set of fixed points of the involution~$\tau$
and to the set of real points of the corresponding real algebraic curve.
If not empty, the boundary of $P/\tau$ decomposes into pairwise disjoint simple closed smooth curves,
called {\it ovals\/}, see~\cite{N1990a}.
(The second kind of closed curves invariant under the involution~$\tau$, called {\it twists},
are invariant curves which are not pointwise fixed by~$\tau$.)
The topological type of a Klein surface $(P,\tau)$ is determined
by the genus~$g$ of~$P$,
the number~$k$ of connected components of the boundary of~$P/\tau$
and the orientability ($\ve=1$) or non-orientability ($\ve=0$) of~$P/\tau$.
The invariants $(g,k,\ve)$ of a Klein surface $(P,\tau)$ satisfy the conditions
$0\le k\le g$ for $\ve=0$, $1\le k\le g+1$ and $k\equiv g+1~(\mod2)$ for $\ve=1$,
see~\cite{Weichold:1883}.
Moreover, the space of all Klein surfaces with the invariants~$(g,k,\ve)$ is connected,
has dimension~$3g-3$ for $g>1$ and is $K(\pi,1)$, see~\cite{N1975, N1978moduli}.

\myskip
Under an $m$-spin Klein surface we understand a Klein surface $(P,\tau)$
with an $m$-spin structure $(P,e:L\to P)$
and an anti-holomorphic involution $\be:L\to L$
such that $e\circ\be=\tau\circ e$.
Recent work~\cite{OT:2013} shows the connection between real $2$-spin Klein surfaces and Abelian Yang-Mills theory.
Combining the topological invariants $(g,k,\ve)$ for the Klein surface $(P,\tau)$
and $(g,m,\de)$ for the $m$-spin surface $(P,e)$ 
we obtain the topological invariants $(g,k,\ve,m,\de)$ of $(P,\tau,e,\be)$.
In this paper we prove that for any Klein surface $(P,\tau)$ of type $(g,k,\ve)$ with $g\ge2$ the number $N(g,k,\ve,m,\de)$
of $m$-spin Klein surfaces $(P,\tau,e,\be)$ with the Arf invariant~$\de$ only depends on the invariants $(g,k,\ve,m,\de)$.
Moreover, we compute the number $N=N(g,k,\ve,m,\de)$:
\begin{enumerate}[$\bullet$]
\item
For $m\equiv1~(\mod2)$ we prove $N=m^g$ if $g\equiv1~(\mod m)$, $\de=0$ and $N=0$ otherwise.
\item
For $\ve=0$, $m\equiv0~(\mod2)$ and $k=0$ we prove $N=\frac{m^g}{2}$ if $g\equiv1~(\mod\frac{m}{2})$ and $N=0$ otherwise.
\item
For $\ve=0$, $m\equiv0~(\mod2)$ and $k\ge1$ we prove $N=m^g\cdot 2^{k-2}$ if $g\equiv1~(\mod\frac{m}{2})$ and $N=0$ otherwise.
\item
For $\ve=1$ and $m\equiv0~(\mod4)$ we prove $N=m^g\cdot 2^{k-2}$ if $g\equiv1~(\mod\frac{m}{2})$ and $N=0$ otherwise.
\item
For $\ve=1$, $m\equiv2~(\mod4)$ and $\de=0$ we prove $N=\frac{m^g}{2}(2^{k-1}+1)$ if $g\equiv1~(\mod\frac{m}{2})$ and $N=0$ otherwise.
\item
For $\ve=1$, $m\equiv2~(\mod4)$ and $\de=1$ we prove $N=\frac{m^g}{2}(2^{k-1}-1)$ if $g\equiv1~(\mod\frac{m}{2})$ and $N=0$ otherwise.
\end{enumerate}
The special case $m=2$ was studied in~\cite{N1990b, N1999,Nbook}.
The cases when $P$ is a sphere or a torus require different methods.


\myskip
As an application of the results of this paper,
a complete list of topological invariants of higher spin Klein surfaces and a description of their moduli spaces
will be studied in the forthcoming paper~\cite{NP:2015b}.

\myskip
Our investigation of $m$-spin Klein surfaces is based on $m$-Arf functions.
An $m$-Arf function is a function on the set $\piorbO(P)$ of oriented simple closed curves on~$P$ with values in~$\z/m\z$
which satisfies certain geometric properties.
It can also be interpreted as the monodromy of a natural connection on the $m$-spin bundle.
According to~\cite{NP:2009}, $m$-Arf functions are in $1$-to-$1$ correspondence with $m$-spin Riemann surfaces.


\myskip
The paper is organised as follows:
In sections~\ref{sec-Aut} and~\ref{sec-klein-spin} we extend the constructions from~\cite{NP:2009} to  Klein surfaces.
We prove that $m$-spin Klein surfaces correspond to $m$-Arf functions which satisfy the conditions
\begin{enumerate}[$\bullet$]
\item
$\arf(\tau c)=-\arf(c)\quad\text{for all}~c\in\piorbO(P)$;
\item
$\arf(c)=0\quad\text{for any twist}~c\in\piorbO(P)$.
\end{enumerate}
In section~\ref{sec-classification} we prove our main theorems.

\myskip
The second author is grateful to the Isaac Newton Institute in Cambridge, where part of this work was done,
for its hospitality and support.
We would like to thank the referee for their valuable remarks and suggestions.

\section{Automorphisms of the Hyperbolic Plane}

\label{sec-Aut}

\subsection{Standard Coverings of the Group of Automorphisms}

Let $G=\Aut(\hyp)$ be the full isometry group of the hyperbolic plane~$\hyp$.
Here our model of the hyperbolic plane is the upper half-plane in $\c$.
This group has two connected components,
the group $G^+=\Aut_+(\hyp)$ of all orientation-preserving isometries of~$\hyp$
and the coset $G^-=\Aut_-(\hyp)$ of all orientation-reversing isometries of~$\hyp$.
Let $e_G$ be the identity element in~$G$.
Let $j\in G^-$ be the reflection in the imaginary axis, $j(z)=-\bar z$.
Then $G^-=j\cdot G^+$.

\begin{mydef}
\label{def-Gm}
Let $\pi:\Gm\to G$ be the Lie group $m$-fold covering of~$G$ given by $\Gm=\Pm\cup\Nm$ with
\begin{align*}
  \Pm=\left\{(g,\de)\in G^+\times\Hol(\hyp,\c^*)\st \de^m=\frac{d}{dz}g\right\},\\
  \Nm=\left\{(g,\de)\in G^-\times\overline{\Hol}(\hyp,\c^*)\st \de^m=\frac{d}{d\bar z}g\right\},
\end{align*}
and the product of elements $(g_1,\de_1)$ and~$(g_2,\de_2)$ in~$\Gm$ given by
$$
  (g_2,\de_2)\cdot(g_1,\de_1)
  =
  \left\{
         \begin{aligned}
            (g_2\circ g_1,(\de_2\circ g_1)\cdot\de_1)&\quad\text{if}\quad(g_2,\de_2)\in\Pm,\\
            (g_2\circ g_1,(\de_2\circ g_1)\cdot\bar{\de_1})&\quad\text{if}\quad(g_2,\de_2)\in\Nm.
         \end{aligned} 
  \right.
$$
The identity element of~$\Gm$ is $e_{\Gm}=(e_G,1)$,
where the second component is the constant function $z\mapsto 1$.
\end{mydef}

\myskip
The centre of~$G^+=\Aut_+(\hyp)$ is trivial, $Z(G^+)=\{e_G\}$.
Hence the centre of~$\Pm$ is contained in $\pi^{-1}(e_G)$, 
the set of elements of the form $(e_G,\exp(2\pi i k/m))$, $k=0,1,\dots,m-1$,
where the second component is the constant function $z\mapsto\exp(2\pi i k/m)$.
From the group law it follows that all such elements belong to the centre,
hence the centre is cyclic of order~$m$:
$$Z(\Pm)=\pi^{-1}(e_G)\cong\z/m\z.$$
Then $U=(e_G,\exp(2\pi i/m))$ is a generator of the centre,
$$Z(\Pm)=\<U\>=\{e_{\Gm},U,U^2,\dots,U^{m-1}\}.$$

\myskip
Let $J\in\Nm$ be a pre-image of the reflection~$j$.
Then $\Pm=J\cdot\Nm$.

\begin{prop}
\label{J-squared}
For a pre-image~$J\in\Nm$ of~$j$ we have~$J^2=e_{\Gm}$.
\end{prop}

\begin{proof}
The element~$J$ must of the form~$J=(j,\de)$ with $\de^m=\frac{d}{d\bar z}j=-1$,
i.e.\ $\de:\hyp\to\c^*$ is a constant function with~$\de^m=-1$.
Hence $J^2=(j,\de)\cdot(j,\de)=(j\circ j,(\de\circ j)\cdot\bar{\de})=(e_G,|\de|^2)=(e_G,1)=e_{\Gm}$.
\end{proof}

\begin{rem}
The Lie group~$G^+$ is connected and has an infinite cyclic fundamental group,
hence the Lie group $m$-fold covering~$\Pm$ of~$G^+$ is unique up to an isomorphism.
The Lie group~$G=G^+\cup G^-$ is not connected and could have several non-isomorphic Lie group $m$-fold coverings.
In fact we will see (compare with remark after Proposition~\ref{J-action})
that for odd~$m$ there is only one $m$-fold covering up to an isomorphism,
while for even~$m$ there are two non-isomorphic $m$-fold coverings
with pre-images of all reflections having order~$2$ or~$4$ respectively.
To describe $m$-spin bundles on Klein surfaces we will use the former covering
which we described explicitly in Definition~\ref{def-Gm}.
\end{rem}

\myskip
Elements of~$G^+$ can be classified with respect to the fixed point behavior of their action on~$\hyp$.
An element is called {\it hyperbolic\/} if it has two fixed points,
which lie on the boundary $\dd\hyp=\r\cup\{\infty\}$ of~$\hyp$.
A hyperbolic element with fixed points~$\al$, $\be$ in~$\r$ is of the form
$$\tau_{\al,\be}(\la):z\mapsto\frac{(\la\al-\be)z-(\la-1)\al\be}{(\la-1)z+(\al-\la\be)},$$
where~$\la>0$.
One of the fixed points of a hyperbolic element is attracting, the other fixed point is repelling.
The {\it axis\/} of a hyperbolic element $g$ is the geodesic between the fixed points of~$g$,
oriented from the repelling fixed point to the attracting fixed point.
The axis of a hyperbolic element is preserved by the element.
The map $\la\mapsto\tau_{\al,\be}(\la)$ defines a homomorphism $\r_+\to G$
(with respect to the multiplicative structure on $\r_+$).
We have $(\tau_{\al,\be}(\la))^{-1}=\tau_{\al,\be}(\la^{-1})=\tau_{\be,\al}(\la)$.

\myskip
An element is called {\it parabolic\/} if it has one fixed point, which is on the boundary~$\dd\hyp$.
A parabolic element with real fixed point $\al$ is of the form
$$\pi_{\al}(\la):z\mapsto\frac{(1-\la\al)z+\la\al^2}{-\la z+(1+\la\al)}.$$
The map $\la\mapsto\pi_{\al}(\la)$ defines a homomorphism $\r\to G$ (with respect to the additive structure on $\r$).
We have $(\pi_{\al}(\la))^{-1}=\pi_{\al}(-\la)$.

\myskip
An element that is neither hyperbolic nor parabolic is called {\it elliptic\/}.
It has one fixed point in $\hyp$.
Given a base-point $x\in\hyp$ and a real number $\varphi$,
let $\rho_x(\varphi)\in G$ denote the rotation through angle $\varphi$ counter-clockwise about the point $x$.
Any elliptic element is of the form $\rho_x(\varphi)$, where $x$ is the fixed point.
Thus we obtain a $2\pi$-periodic homomorphism $\rho_x:\r\to G$ (with respect to the additive structure on $\r$).

\myskip
Elements of $\Pm$ can be classified with respect to the fixed point behavior
of the action of their image in~$G^+$ on $\hyp$. 
We say that an element of~$\Gm$ is {\it hyperbolic\/}, {\it parabolic\/} resp.\ {\it elliptic\/}
if its image in~$G^+$ has this property.

\myskip
The homomorphisms
$$\tau_{\al,\be}:\r_+\to G,\quad \pi_{\al}:\r\to G,\quad \rho_x:\r\to G$$
define one-parameter subgroups in the group~$G$.
Each of these homomorphisms
lifts to a unique homomorphism into the $m$-fold cover:
$$T_{\al,\be}:\r_+\to\Gm,\quad P_{\al}:\r\to\Gm,\quad R_x:\r\to\Gm.$$
The elements $T_{\al,\be}(\la)$, $P_{\al}(\la)$ and~$R_x(\xi)$
are called {\it hyperbolic\/}, {\it parabolic\/} or {\it elliptic\/} respectively.

\myskip
A simple computation shows that for $x=i\in\hyp$ we obtain
$$R_x(2\pi)=(e_G,\exp(2\pi i/m))=U.$$
Hence
$$R_x(2\pi k)=R_x(2\pi)^k=U^k$$
for $x=i\in\hyp$ and any integer $k$.
Since $\rho_x(2\pi k)=\id$ for any integer $k$,
it follows that the lifted element $R_x(2\pi k)$ belongs to $\pi^{-1}(e_G)=Z(\Pm)$.
Note that the element $R_x(2\pi k)$ depends continuously on $x$.
But the fibre $\pi^{-1}(e_G)$ is discrete, so the element $R_x(2\pi k)$ must remain constant,
thus it does not depend on~$x$.
We obtain $R_x(2\pi k)=U^k$ for any integer $k$.

\myskip
The following identities are easy to check geometrically:

\begin{prop}
\label{j-action}
We have
$j\tau_{\al,\be}(\la)j^{-1}=\tau_{-\al,-\be}(\la)$,
$j\pi_{\al}(\la)j^{-1}=\pi_{-\al}(-\la)$,
$j\rho_x(t)j^{-1}=\rho_{-\bar x}(-t)$.
In particular
$j\tau_{0,\infty}(\la)j^{-1}=\tau_{0,\infty}(\la)$,
$j\pi_0(\la)j^{-1}=\pi_0(-\la)$,
$j\rho_i(t)j^{-1}=\rho_i(-t)$.
\end{prop}

Lifting these identities into~$\Gm$ we obtain the following proposition:

\begin{prop}
\label{J-action}
\begin{enumerate}[1)]
\item
We have
\begin{align*}
  &JT_{\al,\be}(\la)J^{-1}=J^{-1}T_{\al,\be}(\la)J=T_{-\al,-\be}(\la),\\
  &JP_{\al}(\la)J^{-1}=J^{-1}P{\al}(\la)J=P_{-\al}(-\la),\\
  &JR_x(t)J^{-1}=J^{-1}R_x(t)J=R_{-\bar x}(-t).
\end{align*}
\item
In particular
\begin{align*}
  &JT_{0,\infty}(\la)J^{-1}=J^{-1}T_{0,\infty}(\la)J=T_{0,\infty}(\la),\\
  &JP_0(\la)J^{-1}=J^{-1}P_0(\la)J=P_0(-\la),\\
  &JR_i(t)J^{-1}=J^{-1}R_i(t)J=R_i(-t),\\
  &JUJ^{-1}=J^{-1}UJ=U^{-1}.
\end{align*}
\end{enumerate}
\end{prop}

\begin{proof}
\begin{enumerate}[1)]
\item
The identity $j\rho_x(t)j^{-1}=\rho_{-\bar x}(-t)$
implies that the paths $t\mapsto JR_x(t)J^{-1}$ and $t\mapsto R_{-\bar x}(-t)$ in~$\Gm$
have the same projection in~$G$ and coincide at~$t=0$, thus
$$JR_x(t)J^{-1}=R_{-\bar x}(-t).$$
The proofs of the other identities are similar.
\item
The proofs are straightforward.
For the last identity recall that $U=R_i(2\pi)$ and $U^{-1}=R_i(-2\pi)$.
\end{enumerate}
\end{proof}

\begin{rem}
In Proposition~\ref{J-squared} we proved that $J^2=e$ using the explicit description of~$\Gm$ in Definition~\ref{def-Gm}
(compare also with the remark after Proposition~\ref{J-squared}).
Note that our proof of Proposition~\ref{J-action} works for any Lie group $m$-fold covering of~$G$,
not just for the one described in~\ref{def-Gm}.
If we forget about Proposition~\ref{J-squared},
we can use Proposition~\ref{J-action} to derive some information about~$J^2$.
For the reflection~$j$ we have $j^2=e$, hence $J^2$ is in the pre-image of~$e$, i.e.\ $J^2=U^q$ for some integer~$q$.
The identities $JUJ^{-1}=U^{-1}$ and $J^{-1}UJ=U^{-1}$ imply $JU=U^{-1}J$ and $UJ=JU^{-1}$.
We have $J^3=J^2J=U^qJ$ and, using $UJ=JU^{-1}$, we obtain $J^3=U^qJ=JU^{-q}$.
On the other hand $J^3=JJ^2=JU^q$.
Thus $JU^{-q}=JU^q$ and therefore $U^{2q}=e$.
For odd~$m$ this is only possible for $q\equiv0~(\mod m)$, hence $J^2=e$,
while for even~$m$ we could have $q\equiv0~(\mod m)$ and hence $J^2=e$ or $q\equiv m/2~(\mod m)$ and hence $J^2=U^{m/2}$, $J^4=e$.
\end{rem}

\subsection{Level function}

\label{subsec-level}

\begin{mydef}
Let $\De$ be the set of all elliptic elements of order~$2$ in~$G^+$.
Let $\Xi$ be the complement of the set~$\De$ in~$G^+$, i.e.\ $\Xi=G^+\backslash\De$.
There exists a homeomorhism~$G^+\to\s^1\times\c$
such that $\De$ corresponds to~$\{*\}\times\c$
and $\Xi=G^+\backslash\De$ corresponds to $(\s^1\backslash\{*\})\times\c$
(see, for example, \cite{JN}).
From this description it follows in particular that the subset~$\Xi$ is simply connected.
The pre-image~$\tilde\Xi\subset\Pm$ of~$\Xi$ consists of $m$ connected components,
each of which is homeomorphic to~$\Xi$.
Each connected component of the subset~$\tilde\Xi$ contains one and only one element of
$$\pi^{-1}(e_G)=Z(\Pm)=\{e_{\Gm},U,\dots,U^{m-1}\}.$$
Let $\tilde\Xi_k$ be the connected component of~$\tilde\Xi$ that contains $U^k$.
Let $a$ be an element of~$\Pm$.
For $a\in\tilde\Xi_k$ we set $\levm(a)=k$.
Any $a\not\in\tilde\Xi$ can be written as $a=R_x(\pi)\cdot U^k$ for some $x\in\hyp$ and some integer~$k$. 
We set $\levm(R_x(\pi)\cdot U^k)=k$ for integer~$k$.
We call the function $\levm:\Pm\to\z/m\z$ the {\it level function}.
We say that $a$ is {\it at the level\/}~$k$ if $\levm(a)=k$.
\end{mydef}

\begin{rem}
Any hyperbolic or parabolic element in~$\Pm$
is of the form $T_{\al,\be}(\la)\cdot U^k$ or $P_{\al}(\la)\cdot U^k$ respectively.
For elements written in this form we have
$$\levm(T_{\al,\be}(\la)\cdot U^k)=k,\quad\levm(P_{\al}(\la)\cdot U^k)=k.$$
\end{rem}





\begin{prop}
\label{prop-conj}
For any elements~$B$ and $C$ in~$\Pm$ we have
$$\levm(BCB^{-1})=\levm(C).$$
\end{prop}






\begin{prop}
\label{J-conj}
We have $\levm(JCJ)=-\levm(C)$ for any hyperbolic or parabolic element~$C$ in~$\Pm$.
\end{prop}

\begin{proof}
Hyperbolic and parabolic elements~$C$ of~$\Pm$
are of the form $T_{\al,\be}(\la)\cdot U^k$ and $P_{\al}(\la)\cdot U^k$ respectively.
According to Proposition~\ref{J-action}
we have $JT_{\al,\be}(\la)J=T_{-\al,-\be}(\la)$, $JP_{\al}(\la)J=P_{-\al}(-\la)$ and $JUJ=U^{-1}$,
hence $J(T_{\al,\be}(\la)\cdot U^k)J=T_{-\al,-\be}(\la)\cdot U^{-k}$
and $J(P_{\al}(\la)\cdot U^k)J=P_{-\al}(-\la)\cdot U^{-k}$.
\end{proof}

\begin{prop}
\label{Nm-conj}
We have $\levm(FCF^{-1})=-\levm(C)$ for any hyperbolic or parabolic element~$C$ in~$\Pm$ and any element~$F$ in~$\Nm$.
\end{prop}

\begin{proof}
We can write the element~$F\in\Nm$ as $F=A\cdot J$ for some~$A\in\Pm$, hence $FCF^{-1}=A(JCJ)A^{-1}$.
According to Proposition~\ref{prop-conj} we have $\levm(A(JCJ)A^{-1})=\levm(JCJ)$
and according to Proposition~\ref{J-conj} we have $\levm(JCJ)=-\levm(C)$.
\end{proof}

\section{Higher Spin Klein Surfaces and Real Arf Functions}

\label{sec-klein-spin}

\subsection{Higher Spin Riemann Surfaces and Lifts of Fuchsian Groups}

\begin{mydef}
Consider a torsion-free Fuchsian group~$\Ga$ and the corresponding hyperbolic Riemann surface~$P=\hyp/\Ga$.
Let $L\to P$ be a complex line bundle on~$P$ and $E\to\hyp$ the induced complex line bundle over $\hyp$.
With respect to a trivialization $E\simeq\hyp\times\c$ of the bundle $E$,
the action of $\Ga$ on $E$ is given by
$$g\cdot(z,t)=(g(z),\de(g,z)\cdot t),$$
where $\de:\Ga\times\h\to\c^*$ is a map
such that the function $\de_g:\hyp\to\hyp$ given be $\de_g(z)=\de(g,z)$ is holomorphic for any $g\in\Ga$ and
$$\de_{g_2\cdot g_1}=(\de_{g_2}\circ g_1)\cdot\de_{g_1}$$
for any $g_1,g_2\in\Ga$.
The map $\de$ is called the {\it transition map} of the bundle $L\to P$ with respect to the given trivialization.
\end{mydef}

\begin{rem}
In particular, if $L$ is the cotangent bundle of the surface $P$,
then the transition map can be chosen so that $\de_g=(g')^{-1}$.
If $L$ is the tangent bundle of the surface $P$,
then the transition map can be chosen so that $\de_g=g'$.
Let $L_1\to P$, $L_2\to P$ be two complex line bundles over a Riemann surface $P$,
and let $\de_1$ resp.\ $\de_2$ be their transition maps,
then $\de_1\cdot\de_2$ is a transition map of the bundle $L_1\otimes L_2\to P$.
In particular, if $\de$ is the transition map of the bundle $L\to P$,
then $\de^m$ is a transition map of the bundle $L^m=L\otimes\cdots\otimes L\to P$
(with respect to the induced trivialization).
\end{rem}

\begin{mydef}
An {\it $m$-spin structure} on a Riemann surface~$P$
is a transition map $\de$ of a complex line bundle $L\to P$
that satisfies the condition $\de_g^m=(g')^{-1}$,
\ie the induced transition map $\de^m$ of the bundle $L^m\to P$
coincides with the transition map of the cotangent bundle of~$P$.
\end{mydef}



\begin{rem}
A complex line bundle $L\to P$ is said to be {\it $m$-spin}
if the bundle $L^m\to P$ is isomorphic to the cotangent bundle of~$P$.
For a compact Riemann surface $P$
there is a 1-1-correspondence between $m$-spin structures on~$P$ and $m$-spin bundles over~$P$.
\end{rem}


\begin{rem}
For $m=2$ we obtain the classical notion of a spin bundle.
\end{rem}

\begin{mydef}
Let $\Ga$ be a Fuchsian group.
A {\it lift of the Fuchsian group}~$\Ga$ into $\Pm$ is a subgroup $\Ga^*$ of $\Pm$
such that the restriction of the covering map $\Pm\to G^+$ to $\Ga^*$ is an isomorphism $\Ga^*\to\Ga$.
\end{mydef}

The following result was proved in~\cite{NP:2005,NP:2009}:

\begin{thm}
\label{thm-corresp-lifts-bundles}
Let $\Ga$ be a Fuchsian group without elliptic elements.
There is a 1-1-correspondence between the lifts of~$\Ga$ into the $m$-fold cover of~$\Aut_+(\hyp)$
and $m$-spin bundles on the Riemann surface $\hyp/\Ga$.
\end{thm}

We will sketch the proof here:
A lift of~$\Ga$ is of the form 
$$\Ga^*=\{(g,\de_g)\st g\in\Ga,~\de_g\in\Hol(\hyp,\c^*),~\de_g^m=\frac{d}{dz}g\}.$$
The corresponding $m$-spin bundle $e_{\Ga^*}:L_{\Ga^*}\to P=\hyp/\Ga$ is of the form
$$L_{\Ga^*}=(\hyp\times\c)/\Ga^*\to\hyp/\Ga=P,$$
where the action of~$\Ga^*$ on~$\hyp\times\c$ is given by
$$(g,\de_g)\cdot(z,x)=(g(z),\de_g(z)\cdot x).$$
Every $m$-spin bundle on $P=\hyp/\Ga$ is obtained as $e_{\Ga^*}$ for some lift $\Ga^*$ of~$\Ga$.

\begin{rem}
A more general correspondence between a Fuchsian group~$\Ga$ (with or without elliptic elements)
and $m$-spin bundles on the orbifold $\hyp/\Ga$
was established in~\cite{NP:2011, NP:2013}.
\end{rem}

\subsection{Lifts of Fuchsian Groups and Arf Functions}

Lifts of a Fuchsian group~$\Ga$ into $\Pm$ can be described by means of associated $m$-Arf functions,
certain functions on the space of homotopy classes of simple closed curves on~$P=\hyp/\Ga$ with values in~$\z/m\z$
described by simple geometric properties.

\begin{mydef}
\label{def-m-arf}
Let $\Ga$ be a Fuchsian group that consists of hyperbolic elements.
Let the corresponding Riemann surface $P=\hyp/\Ga$ be a compact surface with finitely many holes.
Let $\piorb(P)=\piorb(P,p)$ be the fundamental group of~$P$ with respect to a point~$p$.
We denote by $\piorbO(P)$ the set of all non-trivial elements of $\piorb(P,p)$
that can be represented by simple closed curves.
An {\it $m$-Arf function\/} is a function
$$\arf:\piorbO(P)\to\z/m\z$$
satisfying the following conditions
\begin{enumerate}[1.]
\item
\label{arf-prop-conj}
$\arf(bab^{-1})=\arf(a)$ for any elements~$a,b\in\piorbO(P)$,
\item
\label{arf-prop-inv}
$\arf(a^{-1})=-\arf(a)$ for any element~$a\in\piorbO(P)$,
\item
\label{arf-prop-cross}
$\arf(ab)=\arf(a)+\arf(b)$
for any
elements~$a$ and~$b$ which can be represented by a pair of simple closed curves in $P$
intersecting at exactly one point~$p$ with
intersection number not equal to zero.
\item
\label{arf-prop-neg-seqset}
$\arf(ab)=\arf(a)+\arf(b)-1$
for any elements~$a,b\in\piorbO(P)$ such that the element~$ab$ is in~$\piorbO(P)$
and the elements~$a$ and~$b$ can be represented by a pair of simple closed curves in $P$
intersecting at exactly one point~$p$
with intersection number equal to zero
and placed in a neighbourhood of the point~$p$ as shown in Figure~\ref{fig-neg-pair}.
\begin{figure}[h]
  \begin{center}
    \forcehmode
      \bgroup
        \beginpicture
          \setcoordinatesystem units <25 bp,25 bp>
          \multiput {\phantom{$\bullet$}} at -2 0 2 2 /
          \plot 0 0 -2 1 /
          \arrow <7pt> [0.2,0.5] from 0 0 to -1 0.5
          \plot 0 0 -1 2 /
          \arrow <7pt> [0.2,0.5] from -1 2 to -0.5 1
          \plot 0 0 1 2 /
          \arrow <7pt> [0.2,0.5] from 0 0 to 0.5 1
          \plot 0 0 2 1 /
          \arrow <7pt> [0.2,0.5] from 2 1 to 1 0.5
          \put {$a$} [br] <0pt,2pt> at -2 1
          \put {$a$} [br] <0pt,2pt> at -1 2
          \put {$b$} [bl] <0pt,2pt> at 2 1
          \put {$b$} [bl] <0pt,2pt> at 1 2
          \put {$\bullet$} at 0 0
          \put {$p$} [l] <5pt,-3pt> at 0 0 
        \endpicture
      \egroup
  \end{center}
  \caption{$\arf(ab)=\arf(a)+\arf(b)-1$}
  \label{fig-neg-pair}
\end{figure}
\end{enumerate}
\end{mydef}

\begin{rem}
In the case $m=2$ there is a 1-1-correspondence between the $2$-Arf functions in the sense of Definition~\ref{def-m-arf}
and Arf functions in the sense of~\cite{Nbook}, Chapter~1, Section~7 and \cite{N1991}.
Namely, a function $\arf:\piorbO(P)\to\z/2\z$ is a $2$-Arf function
if and only if $\om=1-\arf$ is an Arf function in the sense of~\cite{Nbook}.
\end{rem}

\myskip\noindent
Higher Arf functions were introduced in~\cite{NP:2005, NP:2009}, where the following result was shown:

\begin{thm}
\label{thm-corresp-lifts-arf}
There is a 1-1-correspondence between the lifts of~$\Ga$ into $\Pm$ and $m$-Arf functions on $P=\hyp/\Ga$.
\end{thm}

We will sketch the construction here:
Let $\Psi:\hyp\to P$ be the natural projection.
Choose $q\in\Psi^{-1}(p)$ and let $\Phi:\lat\to\piorb(P)$ be the induced isomorphism.
Consider a lift~$\Ga^*$ of~$\Ga$ into~$\Gm$.
Let $\levm$ be the level function introduced in section~\ref{subsec-level}.

If $\hsi_{\Ga^*}:\piorb(P)\to\z/m\z$ is a function such that the following diagram commutes
$$
  \begin{CD}
   \Ga         @>{\cong}>> \Ga^*         \\
   @V{\Phi}VV  @VV{\levm|_{\Ga^*}}V      \\
   \piorb(P) @>{\hsi_{\Ga^*}}>> \z/m\z,  \\
  \end{CD}
$$
then the function $\arf_{\Ga^*}=\hsi_{\Ga^*}|_{\piorbO(P)}$ is an $m$-Arf function,
the {\it $m$-Arf function associated to the lift $\Ga^*$}.
Every $m$-Arf function is obtained as $\arf_{\Ga^*}$ for some lift $\Ga^*$ of~$\Ga$.

The composite mapping $e_{\Ga^*}\mapsto\Ga^*\mapsto\arf_{\Ga^*}$
establishes a 1-1-cor\-res\-pon\-dence between $m$-spin bundles on $P=\hyp/\Ga$, lifts of the Fuchsian group~$\Ga$
and $m$-Arf functions on~$P$.

\subsection{Klein Surfaces and Real Fuchsian Groups}

\begin{mydef}
A {\it Klein surface} (or a {\it non-singular real algebraic curve\/})
is a topological surface with a maximal atlas whose transition maps are {\it dianalytic},
i.e.\ either holomorphic or anti-holomorphic.
A {\it homomorphism\/} between Klein surfaces is a continuous mapping which is dianalytic in local charts.
\end{mydef}

\myskip
For more information on Klein surfaces, see~\cite{AllingGreenleaf:1971,N1990a}.

\myskip
Let us consider pairs~$(P,\tau)$,
where $P$ is a compact Riemann surface and $\tau:P\to P$ is an anti-holomorphic involution on~$P$.
For each such pair~$(P,\tau)$ the quotient $P/\<\tau\>$ is a Klein surface.
Each isomorphism class of Klein surfaces contains a surface of the form $P/\<\tau\>$.
Moreover, two such quotients $P_1/\<\tau_1\>$ and $P_2/\<\tau_2\>$ are isomorphic as Klein surfaces
if and only if there exists a biholomorphic map~$\psi:P_1\to P_2$ such that $\psi\circ\tau_1=\tau_2\circ\psi$,
in which case we say that the pairs $(P_1,\tau_1)$ and $(P_2,\tau_2)$ are {\it isomorphic}.
Hence from now on we will consider pairs $(P,\tau)$ up to isomorphism instead of Klein surfaces.

\myskip
The category of such pairs $(P,\tau)$ is isomorphic to the category of real algebraic curves
(see~\cite{AllingGreenleaf:1971}),
where fixed points of~$\tau$ (i.e.\ boundary points of the corresponding Klein surface)
correspond to real points of the real algebraic curve.

\myskip
For example a non-singular plane real algebraic curve given by an equation $F(x,y)=0$ 
is the set of real points of such a pair~$(P,\tau)$,
where $P$ is the normalisation and compactification of the surface $\{(x,y)\in\c^2\st F(x,y)=0\}$
and $\tau$ is given by the complex conjugation, $\tau(x,y)=(\bar x,\bar y)$.

\myskip
The set of fixed points of the involution~$\tau$ is called the {\it set of real points\/} of~$(P,\tau)$
and denoted by~$P^{\tau}$.
We say that $(P,\tau)$ is {\it separating\/} or {\it of type~I\/} if the set~$P\backslash P^{\tau}$ is not connected,
otherwise we say that it is {\it non-separating\/} or {\it of type~II\/}.

\myskip
A {\it non-Euclidean crystallographic group\/} or {\it NEC group\/} is a discrete subgroup of~$\Aut(\hyp)$.
The classification of NEC groups was first considered in~\cite{Wilkie:1966}, \cite{Macbeath:1967}.
For more information on NEC groups see~\cite{BEGG} and references therein.
All Klein surfaces can be constructed from real Fuchsian groups, a special kind of NEC groups.

\begin{mydef}
A {\it real Fuchsian group\/} is a NEC group~$\hGa$
such that the intersection $\hGa^+=\hGa\cap\Aut_+(\hyp)$ is a Fuchsian group consisting of hyperbolic automorphisms,
$\hGa\ne\hGa^+$ and the quotient $P=\hyp/\hGa^+$ is a compact surface.
\end{mydef}

Let $\hGa$ be a real Fuchsian group.
Let $\hGa^{\pm}=\hGa\cap\Aut_{\pm}(\hyp)$, $P_{\hGa}=\hyp/\hGa^+$ and let $\Phi:\hyp\to P_{\hGa}$ be the natural projection.
Then for any automorphism~$g\in\hGa^-$,
the map~$\tau_{\hGa}=\Phi\circ g\circ\Phi^{-1}$ is an anti-holomorphic involution of~$P_{\hGa}$.
Thus a real Fuchsian group~$\hGa$ defines the Klein surface $[\hGa]=(P_{\hGa},\tau_{\hGa})$.
It is not hard to see that any Klein surface is obtained this way (see~\cite{Nbook,N1975,N1978moduli}).

\begin{prop}
\label{prop-induced-inv}
Let $\hGa$ be a real Fuchsian group and $[\hGa]=(P_{\hGa},\tau_{\hGa})$ the corresponding Klein surface as defined above.
The anti-holomorphic involution $\tau=\tau_{\hGa}$ on $P_{\hGa}$
induces an involution $\tau=(\tau_{\hGa})_*$ on $\piorb(P_{\hGa})\cong\hGa^+$.
The induced involution satisfies
$$\tau(c)=fcf^{-1}$$
for every $c\in\hGa^+$ and $f\in\hGa^-$.
\end{prop}

\begin{proof}
Let $f\in\hGa^-$.
An element $c\in\hGa^+$ corresponds to the closed curve $[\Phi(\ell_c)]\in\piorb(P_{\hGa})$,
where $\ell_c$ is the axis of~$c$ and $\Phi:\hyp\to P_{\hGa}$ is the natural projection.
The image of~$c$ under the induced involution~$\tau$ corresponds to the closed curve
$$
  [\tau_{\hGa}(\Phi(\ell_c))]
  =[(\Phi\circ f\circ\Phi^{-1})(\Phi(\ell_c))]
  =[\Phi(f(\ell_c))].
$$
It is easy to see geometrically that $f(\ell_c)$ is the axis of~$fcf^{-1}$,
hence $\tau(c)=fcf^{-1}$.
\end{proof}

\subsection{From Lifts of Real Fuchsian Groups to Higher Spin Klein Surfaces}

\begin{mydef}
An {\it $m$-spin bundle on a Klein surface\/}~$(P,\tau)$ is a pair~$(e,\be)$,
where $e:L\to P$ is an $m$-spin bundle on~$P$ and $\be:L\to L$ is an anti-holomorphic involution on~$L$
such that $e\circ\be=\tau\circ e$, i.e.\ the following diagram commutes:
$$
  \begin{CD}
   L@>{e}>>P\\
   @V{\be}VV@VV{\tau}V\\
   L@>{e}>>P\\
  \end{CD}
$$
\end{mydef}

\begin{mydef}
\label{def-isomorphic-m-spin-bundles}
Two $m$-spin bundles $(e_1:L_1\to P_1,\be_1)$ and $(e_2:L_2\to P_2,\be_2)$
on Klein surfaces $(P_1,\tau_1)$ and $(P_2,\tau_2)$ are {\it isomorphic\/}
if there exist biholomorphic maps $\varphi_L:L_1\to L_2$ and $\varphi_P:P_1\to P_2$
such that $e_2\circ\varphi_L=\varphi_P\circ e_1$, $\be_2\circ\varphi_L=\varphi_L\circ\be_1$
and $\tau_2\circ\varphi_P=\varphi_P\circ\tau_1$. i.e.\ the obvious diagrams commute:
$$
  \begin{CD}
   L_1@>{\be_1}>>L_1@>{e_1}>>P_1@>{\tau_1}>>P_1\\
   @V{\varphi_L}VV@V{\varphi_L}VV@VV{\varphi_P}V@VV{\varphi_P}V\\
   L_2@>{\be_2}>>L_2@>{e_2}>>P_2@>{\tau_2}>>P_2\\
  \end{CD}
$$
\end{mydef}

\begin{mydef}
A {\it lift of a real Fuchsian group}~$\hGa$
into $\Gm$ is a subgroup~$\hGa^*$ of~$\Gm$
such that the projection $\pi|_{\hGa^*}:\hGa^*\to\hGa$ is an isomorphism.
\end{mydef}

\begin{prop}
To any lift of a real Fuchsian group into the $m$-fold cover~$\Gm$ of~$\Aut(\hyp)$
we can associate an $m$-spin bundle on the corresponding Klein surface.
\end{prop}

\begin{proof}
Consider a lift~$\hGa^*$ of a real Fuchsian group~$\hGa$ into the $m$-fold cover $\Gm$ of~$\Aut(\hyp)$.
The corresponding Fuchsian group is given by $\Ga=\hGa^+=\hGa\cap\Aut_+(\hyp)$,
while $\Ga^*=\hGa^*\cap\Pm$ is the corresponding lift of~$\Ga$ into~$\Pm$.
Let $e_{\Ga_*}:L_{\Ga^*}\to P$ be the corresponding $m$-spin bundle as in Theorem~\ref{thm-corresp-lifts-bundles}
with
$$P=\hyp/\Ga\quad\text{and}\quad L_{\Ga^*}=(\hyp\times\c)/\Ga^*.$$
For any~$(g,\de_g)\in\hGa^*\cap\Nm$, consider the mapping $(z,x)\mapsto(g(z),\de_g(z)\cdot\bar x)$.

If $(z',x')$ and $(z,x)$ correspond to the same point in~$L_{\Ga^*}=(\hyp\times\c)/\Ga^*$, then computation
shows that $(g(z'),\de_g(z')\cdot\bar x')$ and $(g(z),\de_g(z)\cdot\bar x)$ correspond to the same point in~$L_{\Ga^*}$.
Thus the mapping $(z,x)\mapsto(g(z),\de_g(z)\cdot\bar x)$ induces a map~$\be_{\hGa^*}:L_{\Ga^*}\to L_{\Ga^*}$.

It is not hard to check that if we choose two different elements $(g_1,\de_{g_1})$ and $(g_2,\de_{g_2})$ in $\hGa^*\cap\Nm$,
then $(g_1(z),\de_{g_1}(z)\cdot\bar x)$ and $(g_2(z),\de_{g_2}(z)\cdot\bar x)$
correspond to the same point in~$L_{\Ga^*}$.
Thus the map $\be_{\hGa^*}$ does not depend on the choice of the element~$g\in\hGa^*\cap\Nm$.


If we apply $\be_{\hGa^*}$ twice we get
\begin{align*}
  (z,x)
  &\mapsto(g(z),\de_g(z)\cdot\bar x)\\
  &\mapsto(g(g(z)),\de_g(g(z))\cdot\overline{\de_g(z)\cdot\bar x})\\
  &=((g\circ g)(z),((\de_g\circ g)\cdot\bar\de_g)(z)\cdot x)\\
  &=((g\circ g)(z),\de_{g\circ g}(z)\cdot x)\\
  &=(g\circ g)\cdot(z,x).
\end{align*}
We have $g\circ g\in\hGa^*$ since~$g\in\hGa^*$ and we have $g\circ g\in\Pm$ for any~$g\in\Gm$,
hence $g\circ g\in\hGa^*\cap\Pm=\Ga^*$.
Thus $(z,x)$ and $(g\circ g)\cdot(z,x)$ are equal modulo the action of~$\Ga^*$.
We have therefore shown that $\be_{\hGa^*}$ is indeed an involution.
We can now associate with the lift~$\hGa^*$ of the real Fuchsian group~$\hGa$
the $m$-spin bundle $e_{\hGa^*}:=(e_{\Ga^*},\be_{\hGa^*})$.
\end{proof}

\begin{prop}
To any $m$-spin bundle on the Klein surface~$(P,\tau)$ 
we can associate a lift of a real Fuchsian group into the $m$-fold cover~$\Gm$ of~$\Aut(\hyp)$.
\end{prop}

\begin{proof}
Any $m$-spin bundle on $(P,\tau)$ is obtained as $e_{\Ga^*}$ for some lift $\Ga^*$ of~$\Ga$ into~$\Gm$.
Let $(e:L\to P,\be:L\to L)$ be an $m$-spin bundle on $(P,\tau)$.
We have $e\circ\be=\tau\circ e$.
Consider a lift $\tilde\be$ of $\be:L\to L$ to the universal cover $\tilde L=\hyp\times\c$ of~$L$.
Let $\tilde e$ be the projection $\hyp\times\c\to P$. 
The map~$\tilde\be$ is bi-anti-holomorphic, invariant under~$\Ga^*$
and with the property $\tilde e\circ\tilde\be=\tau\circ\tilde e$,
hence $\tilde\be$ is of the form
$$\tilde\be(z,x)=(g(z),f(z,x)),$$
where $g$ is some element of $\hGa^-$ and $f$ is some anti-holomorphic map.
For a fixed~$z$ the map $x\mapsto f(z,x)$ is a bi-anti-holomorphic map $\c\to\c$,
hence $f(z,x)=a(z)\cdot\bar x+b(z)$,
where $a:\hyp\to\c^*$ and $b:\hyp\to\c$ are holomorphic functions.
Since $\be$ is a bundle map, it preserves the zero section of~$L$, hence $b(z)=0$ for all~$z$.
Thus $\tilde\be$ is of the form
$$\tilde\be(z,x)=(g(z),a(z)\cdot\bar x),$$
where $a:\hyp\to\c^*$ is a holomorphic function.
Considering the $m$-fold tensor products, we obtain an anti-holomorphic involution given by
$$\tilde\be^{\otimes m}(z,x)=(g(z),a^m(z)\cdot\bar x)$$
on the cotangent bundle of~$P$,
hence 
$$a^m=\frac{d}{d\bar z}g.$$
Therefore $\tilde g=(g,\de_g)$ with $\de_g=a$ defines a lift of the element~$g$ into~$\Nm$.
The fact that the map $\tilde\be$ is invariant under the action of~$\Ga^*$ on $\hyp\times\c$
implies that the element~$\tilde g=(g,\de_g)$ normalises the lift~$\Ga^*$,
i.e.\ $\tilde g\cdot\Ga^*\cdot\tilde g^{-1}=\Ga^*$.
The fact that $\be$ is an involution implies that the element~$\tilde g=(g,\de_g)$ is of order two.
The fact that $\tilde g\cdot\Ga^*\cdot\tilde g^{-1}=\Ga^*$ and $\tilde g^2=\tilde e$
implies that the subgroup of $\Gm$ generated by $\Ga^*$ and $\tilde g$ is a lift of $\hGa$ into $\Gm$.
\end{proof}

\subsection{Lifts of Real Fuchsian Groups and Real Arf Functions}

A lift~$\hGa^*$ of a real Fuchsian group~$\hGa$ into the $m$-fold cover $\Gm$ of~$\Aut(\hyp)$
induces a lift $\Ga^*=\hGa^*\cap\Pm$ of the Fuchsian group $\Ga=\hGa\cap G^+$ into~$\Pm$,
whence an $m$-Arf function~$\arf_{\hGa^*}$ on~$\hyp/\Ga$.
Let us study the special properties of such $m$-Arf functions.

\begin{lem}
\label{lem-compatible-arf-function}
Let $\hGa$ be a real Fuchsian group,
$\Ga=\hGa^+=\hGa\cap G^+$ the corresponding Fuchsian group,
$[\hGa]=(P=\hyp/\Ga,\tau)$ the corresponding Klein surface
and $\hGa^*$ a lift of~$\hGa$.
Then the induced $m$-Arf function~$\arf=\arf_{\Ga^*}$ on~$P$ has the following property:
$\arf(\tau c)=-\arf(c)$ for any~$c\in\piorbO(P)$.
\end{lem}

\begin{proof}
The anti-holomorphic involution on~$P$ is given by $\tau=\Phi\circ f\circ\Phi^{-1}$,
where $f\in\hGa^-=\hGa\cap G^-$ and $\Phi$ is the natural projection $\hyp\to P$.
The induced involution on $\piorb(P)\cong\Ga\cong\Ga^*$
is given by conjugation by an element of~$(\hGa^*)^-=\hGa^*\cap\Nm$,
which according to Proposition~\ref{Nm-conj} changes the sign of~$\levm$,
hence $\arf(\tau c)=-\arf(c)$ for all~$c\in\piorbO(P)$.
\end{proof}

\begin{mydef}
We call an $m$-Arf function on a Klein surface~$(P,\tau)$ {\it compatible\/} (with the involution~$\tau$)
if $\arf(\tau c)=-\arf(c)$ for any~$c\in\piorbO(P)$.
\end{mydef}

To understand the structure of a Klein surface $(P,\tau)$,
we look at those closed curves which are invariant under the involution~$\tau$.
There are two kinds of invariant curves, depending on whether the restriction of $\tau$ to the invariant curve
is the identity or a "half-turn".

\begin{mydef}
Let $(P,\tau)$ be a Klein surface.
The set~$P^{\tau}$ of fixed points of the involution~$\tau$
decomposes into pairwise disjoint simple closed smooth curves, called {\it ovals}.
\end{mydef}

\begin{mydef}
A {\it twist\/} (or {\it twisted oval\/}) is a simple closed curve in~$P$ which is invariant under the involution~$\tau$
but does not contain any fixed points of~$\tau$.
\end{mydef}

\begin{rem}
A twisted oval is not an oval,
however the corresponding element of $H_1(P)$ is a fixed point of the induced involution
and the corresponding element of $\piorb(P)$ is preserved up to conjugation by the induced involution.
\end{rem}

\begin{lem}
\label{lem-compatible-arf-zero-or-half-m}
Let $\arf$ be a compatible $m$-Arf function on a Klein surface~$(P,\tau)$.
If $m$ is odd, then $\arf$ vanishes on all ovals and all twists.
If $m$ is even, then $\arf(c)$ is either equal to~$0$ or to~$m/2$ for any oval and any twist~$c$.
\end{lem}

\begin{proof}
For any invariant curve $c\in\piorb(P)\cong\Ga$, either an oval or a twist,
we have $\tau c=c$ and therefore $\arf(\tau c)=\arf(c)$.
On the other hand $\arf$ is compatible, hence $\arf(\tau c)=-\arf(c)$ for all~$c$.
Therefore $2\arf(c)=0$ modulo~$m$.
For odd~$m$ this implies $\arf(c)=0$, while for even~$m$ we can have either $\arf(c)=0$ or $\arf(c)=m/2$.
\end{proof}

\myskip
Not all compatible $m$-Arf functions correspond to lifts of real Fuchsian groups.
We will prove now that if an $m$-Arf function corresponds to a lift of a real Fuchsian group,
then stronger conditions on the twists than Lemma~\ref{lem-compatible-arf-zero-or-half-m} are satisfied.

\myskip
For a hyperbolic automorphism~$c\in\Aut_+(\hyp)$
let $\bar c$ be the reflection whose mirror coincides with the axis of~$c$,
let $\sqrt{c}$ be the hyperbolic automorphism such that~$(\sqrt{c})^2=c$
and let $\tilde c=\bar c\sqrt{c}$.
The discussion summarised in section~2.2 of~\cite{Nbook}
(compare with Theorem~\ref{thm-gen} for more details) implies

\begin{lem}
\label{lem-bar-tilde-c}
If $c\in\hGa$ is a hyperbolic element that corresponds to an oval on $P=\hyp/\Ga$,
then $\hGa$ contains the reflection $\bar c$.
If $c\in\hGa$ is a hyperbolic element that corresponds to a twist on $P=\hyp/\Ga$,
then $\hGa$ contains the element $\tilde c=\bar c\sqrt{c}$.
\end{lem}

\begin{lem}
\label{lem-non-special-arf}
Let $\hGa$ be a real Fuchsian group,
$\Ga=\hGa^+=\hGa\cap G^+$ the corresponding Fuchsian group,
$[\hGa]=(P=\hyp/\Ga,\tau)$ the corresponding Klein surface and $\hGa^*$ a lift of~$\hGa$.
Then the induced $m$-Arf function~$\arf=\arf_{\Ga^*}$ on~$P$ vanishes on all twists.
\end{lem}

\begin{proof}
We need to show that the case $m$ even, $c$ a twist, $\arf(c)=m/2$ is not possible.
Let $c$ be a hyperbolic element in $\Ga\cong\piorb(P)$ which corresponds to a twist.
According to Lemma~\ref{lem-bar-tilde-c} the group~$\hGa$ contains the element~$\tilde c=\bar c\sqrt{c}$.
Let $C\in(\hGa^*)^+$ and $\tilde C\in(\hGa^*)^-$ be the lifts of~$c$ and $\tilde c$.
Without loss of generality we can assume that $c=\tau_{0,\infty}(\la)$,
so that $\tilde c=j\cdot\tau_{0,\infty}(\la/2)$.
Then the lift of~$\tilde c$ in~$\hGa^*$ is of the form $JT_{0,\infty}(\la/2)\cdot U^q$ for some integer~$q$.
Using identities from Proposition~\ref{J-action} we obtain that
\begin{align*}
  (J T_{0,\infty}(\la/2) U^q)^2
  &=(J T_{0,\infty}(\la/2) U^q)(J T_{0,\infty}(\la/2) U^q)\\
  &=(J T_{0,\infty}(\la/2) U^q)(U^{-q} J T_{0,\infty}(\la/2))\\
  &=J T_{0,\infty}(\la/2) J T_{0,\infty}(\la/2)
  =J J T_{0,\infty}(\la/2) T_{0,\infty}(\la/2)\\
  &=T_{0,\infty}(\la/2) T_{0,\infty}(\la/2)
  =T_{0,\infty}(\la).
\end{align*}
The element~$T_{0,\infty}(\la)$ is therefore in~$\hGa^*$ and is a pre-image of $(\tilde c)^2=c$,
hence $T_{0,\infty}(\la)$ is the lift of~$c$ in~$\hGa^*$.
By definition of an induced $m$-Arf function, we know that $\arf(c)$ is equal to the value of the level function~$\levm$
on the lift of~$c$ in~$\hGa^*$,
i.e.\ $\arf(c)=\levm(T_{0,\infty}(\la))$.
On the other hand, $\levm(T_{0,\infty}(\la))=0$, see section~\ref{subsec-level}.
Thus we conclude that $\arf(c)=0$.
\end{proof}

\begin{mydef}
\label{def-real-arf}
A {\it real $m$-Arf function} on a Klein surface~$(P,\tau)$ is an $m$-Arf function~$\arf$ on~$P$ such that
\begin{enumerate}[(i)]
\item
$\arf$ is compatible with~$\tau$, i.e.\ $\arf(\tau c)=-\arf(c)$ for any~$c\in\piorbO(P)$.
\item
$\arf$ vanishes on all twists.
\end{enumerate}
\end{mydef}

Lemmas~\ref{lem-compatible-arf-function} and~\ref{lem-non-special-arf} imply:

\begin{thm}
\label{thm-lifts-to-Arf-functions}
Let $\hGa^*$ be a lift of a real Fuchsian group~$\hGa$.
Then the induced $m$-Arf function~$\arf=\arf_{\hGa^*}$
is a real $m$-Arf function on the Klein surface $[\hGa]$.
\end{thm}

\begin{rem}
We slightly change terminology here.
In~\cite{Nbook, N1994} $2$-Arf functions satisfying~(i) were called real Arf functions,
while $2$-Arf functions satisfying~(i) and~(ii) were called non-special real Arf functions.
\end{rem}

\begin{rem}
One can show 
that in the presence of ovals ($P^{\tau}\ne\emptyset$)
compatibility with~$\tau$ implies property~(ii).
However, if $P^{\tau}=\emptyset$ and $m$ is even,
there exist compatible Arf functions which assume the value $m/2$ on all twists.
We will not be interested in these Arf functions since they do not come from real Fuchsian groups.
\end{rem}



\begin{mydef}
Two lifts~$(\hGa^*)_1$ and $(\hGa^*)_2$ of a real Fuchsian group~$\hGa$ are {\it similar\/}
if $(\hGa^*)_1^-=(\hGa^*)_2^-\cdot U^q$ for some integer~$q$.
\end{mydef}

\begin{rem}
Note that for similar lifts~$(\hGa^*)_1$ and $(\hGa^*)_2$ of a real Fuchsian group~$\hGa$
we have $(\hGa^*)_1^+=(\hGa^*)_2^+$.
\end{rem}
 
\begin{thm}
\label{corresp-lifts-arf}
Let $\hGa$ be a real Fuchsian group.
The mapping that assigns to a lift~$\hGa^*$ of~$\hGa$ into~$\Gm$ the $m$-Arf function of~$(\hGa^*)^+$
establishes a 1-1-correspondence between similarity classes of lifts of the real Fuchsian group~$\hGa$
and real $m$-Arf functions on the Klein surface $[\hGa]$.
\end{thm}

\begin{proof}
According to Theorem~\ref{thm-lifts-to-Arf-functions}, we can assign to any lift~$\hGa^*$ of~$\hGa$
a real $m$-Arf function~$\arf=\arf_{\hGa^*}$ on $[\hGa]$.

\myskip
Let $\arf$ be a real $m$-Arf function on $[\hGa]$.
We will show that there exist $m$ lifts of $\hGa$
such that the $m$-Arf function induced by each of the lifts is equal to~$\arf$.
Moreover, these $m$ lifts are similar to each other.

\myskip
According to Theorem~\ref{thm-corresp-lifts-arf} there exists a unique lift~$\Ga^*$ into~$\Pm$ of the group~$\Ga=\hGa^+$
such that the induced $m$-Arf function is equal to~$\arf$.
Choose an element $f$ in~$\hGa^-$.
To uniquely determine a lift of the group $\hGa=\<\Ga,f\>=\Ga\cup f\cdot\Ga$
we need to choose one of the $m$~pre-images of~$f$.
We will show that for any pre-image~$F$ of~$f$ the subset $\Ga^*\cup F\cdot\Ga^*$ is a subgroup and hence a lift of~$\hGa$. 
To this end we need to prove that $F\Ga^*F^{-1}=\Ga^*$ and $F^2\in\Ga^*$.

\myskip
We will start with $F\Ga^*F^{-1}=\Ga^*$:
Let $C\in\Ga^*$ be a pre-image of~$c\in\Ga$, i.e. $\levm(C)=\arf(c)$.
Using Proposition~\ref{prop-induced-inv}, Definition~\ref{def-real-arf} and Proposition~\ref{Nm-conj} we obtain
$$\arf(fcf^{-1})=\arf(\tau c)=-\arf(c)=-\levm(C)=\levm(FCF^{-1}),$$
i.e.\ $FCF^{-1}$ is a pre-image of $fcf^{-1}$ with the same level as the lift of $fcf^{-1}$ in~$\Ga^*$.
Hence $FCF^{-1}$ is the lift of~$fcf^{-1}$ in $\Ga^*$.
Thus we proved that $F\Ga^*F^{-1}\subset\Ga^*$.
Considering $F^{-1}$ and $f^{-1}$ instead of $F$ and $f$
we obtain $F^{-1}\Ga^*F\subset\Ga^*$ and hence $\Ga^*\subset F\Ga^*F^{-1}$.
Therefore $F\Ga^*F^{-1}=\Ga^*$.

\myskip
Now we will show $F^2\in\Ga^*$ in the separating case:
In this case, we can assume without loss of generality that $f=j$.
The lifts of~$f$ are then of the form $F=J\cdot U^q$.
Using identities from Proposition~\ref{J-action} we obtain for any integer~$q$ that
$$
  F^2
  =(JU^q)^2
  =(JU^q)(JU^q)
  =(JU^q)(U^{-q}J)
  =J^2
  =\tilde e.
$$ 

\myskip
Now we will show $F^2\in\Ga^*$ in the non-separating case:
In this case, we can assume without loss of generality that $f=j\tau_{0,\infty}(\la/2)$,
where the element~$\tau_{0,\infty}(\la)$ corresponds to a twist.
The lifts of~$f$ are then of the form $F=JT_{0,\infty}(\la/2)U^q$.
Using identities from Proposition~\ref{J-action} we obtain as in the proof of Lemma~\ref{lem-non-special-arf}
that for any integer~$q$
$$
  F^2
  =(JT_{0,\infty}(\la/2)\cdot U^q)^2
  =T_{0,\infty}(\la).
$$
Since the Arf function~$\arf$ is real, we have
$$\arf(\tau_{0,\infty}(\la))=0=\levm(T_{0,\infty}(\la)),$$
hence the element $F^2=T_{0,\infty}(\la)$ is in~$\Ga^*$.

\myskip
We have now established that for any pre-image~$F$ of~$f$,
the properties $F\Ga^*F^{-1}=\Ga^*$ and $F^2\in\Ga^*$ are satisfied,
hence the subgroup of~$\Gm$ generated by $\Ga^*$ and $F$ is a lift of ~$\hGa$.
\end{proof}

\subsection{From Higher Spin Klein Surfaces to Lifts of Real Fuchsian Groups}

\begin{thm}
\label{corresp-bundles-lifts}
There is a 1-1-correspondence between $m$-spin bundles on Klein surfaces
and similarity classes of lifts of real Fuchsian groups into the $m$-fold cover~$\Gm$ of~$\Aut(\hyp)$.
\end{thm}

\begin{proof}
We will show that similar lifts of a real Fuchsian group~$\hGa$
induce isomorphic $m$-spin bundles on the corresponding Klein surface.
The uniformisations of the anti-holomorphic involutions that correspond to similar lifts of~$\hGa$ are of the form
$\tilde\be_1(z,x)=(g(z),\de_g(z)\cdot\bar x)$ and $\tilde\be_2(z,x)=(g(z),\zeta\cdot\de_g(z)\cdot \bar x)$,
where $\zeta\in\c$, $\zeta^m=1$.
Then taking $\varphi_L(z,x)=(z,\sqrt{\zeta}\cdot x)$ and $\varphi_P=\id_P$
(see Definition~\ref{def-isomorphic-m-spin-bundles}) we can see
that two $m$-spin bundles on the Klein surface are isomorphic:
\begin{align*}
  &\varphi_L(\tilde\be_1(z,x))
  =\varphi_L(g(z),\de_g(z)\cdot\bar x)
  =(g(z),\sqrt{\zeta}\cdot\de_g(z)\cdot\bar x),\\
  &\tilde\be_2(\varphi_L(z,x))
  =\tilde\be_2(z,\sqrt{\zeta}\cdot x)
  =(g(z),\zeta\cdot\de_g(z)\cdot\overline{\sqrt{\zeta}\cdot x})\\
  &=(g(z),\zeta\cdot\sqrt{\bar\zeta}\cdot\de_g(z)\cdot\bar x)
  =(g(z),\sqrt{\zeta}\cdot\de_g(z)\cdot\bar x).\qedhere
\end{align*}
\end{proof}

\bigskip\noindent
Theorems~\ref{corresp-lifts-arf} and~\ref{corresp-bundles-lifts} immediately imply

\begin{thm}
\label{corresp-bundles-arf}
The mapping that assigns to an $m$-spin bundle on a Klein surface the corresponding real $m$-Arf function
establishes a 1-1-correspondence.
\end{thm}



\section{Real Arf Functions}

\label{sec-classification}


\subsection{Topological Invariants of Klein Surfaces}

\label{subsec-topinv-klein}

\begin{mydef}
Given two Klein surfaces $(P_1,\tau_1)$ and $(P_2,\tau_2)$,
we say that they are {\it topologically equivalent\/}
if there exists a homeomorhism~$\phi:P_1\to P_2$ such that $\phi\circ\tau_1=\tau_2\circ\phi$.
\end{mydef}

Recall that a Klein surface $(P,\tau)$ is called non-separating if the set~$P\backslash P^{\tau}$ is connected,
otherwise $(P,\tau)$ is called separating.

\begin{mydef}
The {\it topological type\/} of a Klein surface~$(P,\tau)$ is the triple~$(g,k,\ve)$,
where $g$ is the genus of the Riemann surface~$P$,
$k$ is the number of connected components of the fixed point set $P^{\tau}$ of~$\tau$,
$\ve=0$ if $(P,\tau)$ is non-separating and $\ve=1$ otherwise.
We say that a real Fuchsian group $\hGa$ is of topological type $(g,k,\ve)$
if the corresponding Klein surface is of topological type $(g,k,\ve)$.
\end{mydef}

In this paper we consider hyperbolic compact surfaces, hence $g\ge2$.
The following result of Weichold~\cite{Weichold:1883} gives a classification of Klein surfaces
up to topological equivalence:

\begin{thm}
\label{thm-Weichold}
Two Klein surfaces are topologically equivalent if and only if they are of the same topological type.
A triple $(g,k,\ve)$ is a topological type of some Klein surface if and only if
either $\ve=1$, $1\le k\le g+1$, $k\equiv g+1~(\mod2)$ or $\ve=0$, $0\le k\le g$. 
\end{thm}

\myskip
Any separating Klein surface can be obtained by gluing together a Riemann surface with boundary
with its copy via the identity map along the boundary components.
If we replace the identity map with a half-turn on some of the boundary components,
we obtain a non-separating Klein surface.
Moreover, all non-separating Klein surfaces are obtained this way.

\myskip
We will use the following description of generating sets of Fuchsian groups (see~\cite{NP:2005, NP:2009, Zieschang:book}):

\begin{mydef}
\label{def-standard-basis}
Consider a compact Riemann surface~$P$ of genus~$g$ with $n$~holes
and the corresponding Fuchsian group~$\Ga$ such that $P=\hyp/\Ga$.
Let $p\in P$.
A {\it standard generating set\/} of the fundamental group $\piorb(P,p)\cong\Ga$ is a set
$$(a_1,b_1,\dots,a_g,b_g,c_1,\dots,c_n)$$
of elements in~$\piorb(P,p)$ that can be represented by simple closed curves on~$P$ based at~$p$
$$(\tilde a_1,\tilde b_1,\dots,\tilde a_g,\tilde b_g,\tilde c_1,\dots,\tilde c_n)$$
with the following properties:
\begin{enumerate}[$\bullet$]
\item
The curve $\tilde c_i$ encloses a hole in $P$ for $i=1,\dots,n$.
\item
Any two curves only intersect at the point~$p$.
\item
A neighbourhood of the point $p$ with the curves is homeomorphic to the one shown in Figure~\ref{fig-basis}.
\begin{figure}[h]
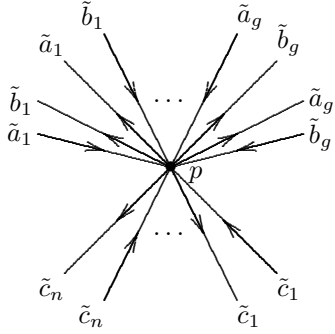

  \begin{center}
    \forcehmode
      \bgroup
        \beginpicture
          \setcoordinatesystem units <25 bp,25 bp>
          \multiput {\phantom{$\bullet$}} at -2 -2 2 2 /
          \plot 0 0 -2 0.5 /
          \arrow <7pt> [0.2,0.5] from -2 0.5 to -1 0.25
          \plot 0 0 -2 1 /
          \arrow <7pt> [0.2,0.5] from 0 0 to -1 0.5
          \plot 0 0 -1.6 1.6 /
          \arrow <7pt> [0.2,0.5] from 0 0 to -0.8 0.8
          \plot 0 0 -1 2 /
          \arrow <7pt> [0.2,0.5] from -1 2 to -0.5 1
          \plot 0 0 1 2 /
          \arrow <7pt> [0.2,0.5] from 1 2 to 0.5 1
          \plot 0 0 1.6 1.6 /
          \arrow <7pt> [0.2,0.5] from 0 0 to 0.8 0.8
          \plot 0 0 2 1 /
          \arrow <7pt> [0.2,0.5] from 0 0 to 1 0.5
          \plot 0 0 2 0.5 /
          \arrow <7pt> [0.2,0.5] from 2 0.5 to 1 0.25
          \plot 0 0 1.6 -1.6 /
          \arrow <7pt> [0.2,0.5] from 1.6 -1.6 to 0.8 -0.8
          \plot 0 0 1 -2 /
          \arrow <7pt> [0.2,0.5] from 0 0 to 0.5 -1
          \plot 0 0 -1 -2 /
          \arrow <7pt> [0.2,0.5] from -1 -2 to -0.5 -1
          \plot 0 0 -1.6 -1.6 /
          \arrow <7pt> [0.2,0.5] from 0 0 to -0.8 -0.8
          \put {$\tilde a_1$} [r] <-2pt,0pt> at -2 0.5
          \put {$\tilde b_1$} [r] <-2pt,0pt> at -2 1
          \put {$\tilde a_1$} [br] <0pt,2pt> at -1.6 1.6
          \put {$\tilde b_1$} [br] <0pt,2pt> at -1 2
          \put {$\tilde a_g$} [bl] <0pt,2pt> at 1 2
          \put {$\tilde b_g$} [bl] <0pt,2pt> at 1.6 1.6
          \put {$\tilde a_g$} [l] <2pt,0pt> at 2 1
          \put {$\tilde b_g$} [l] <2pt,0pt> at 2 0.5
          \put {$\tilde c_1$} [tl] <0pt,-2pt> at 1 -2
          \put {$\tilde c_1$} [tl] <0pt,-2pt> at 1.6 -1.6
          \put {$\tilde c_n$} [tr] <0pt,-2pt> at -1.6 -1.6
          \put {$\tilde c_n$} [tr] <0pt,-2pt> at -1 -2
          \put {$\dots$} at 0 1
          \put {$\dots$} at 0 -1
          \put {$\bullet$} at 0 0
          \put {$p$} [l] <7pt,-3pt> at 0 0 
        \endpicture
      \egroup
  \end{center}
  \caption{Standard generating set}
  \label{fig-basis}
\end{figure}
\item
The system of curves cuts the surface~$P$ into $n+1$ connected components of which $n$ are homeomorphic to an annulus
and one is homeomorphic to a disc and has boundary
$$
  \tilde a_1\tilde b_1\tilde a_1^{-1}\tilde b_1^{-1}\dots\tilde a_g\tilde b_g\tilde a_g^{-1}\tilde b_g^{-1}
  \tilde c_1\dots\tilde c_n.
$$
\end{enumerate}
\end{mydef}

\myskip
We will use the following description of generating sets of real Fuchsian groups given in~\cite{Nbook,N1975,N1978moduli}:

\begin{thm}
\label{thm-gen}

\smallskip
{\bf(Generating sets of real Fuchsian groups)}

Recall that for a hyperbolic automorphism~$c\in\Aut_+(\hyp)$,
$\bar c$ is the reflection whose mirror coincides with the axis of~$c$,
$\sqrt{c}$ is the hyperbolic automorphism such that~$(\sqrt{c})^2=c$
and $\tilde c=\bar c\sqrt{c}$.

\begin{enumerate}[1)]
\item
Let $(g,k,1)$ be a topological type of a Klein surface, i.e.\ $1\le k\le g+1$ and $k\equiv g+1~(\mod2)$. 
Set $n=k$.
Consider a Fuchsian group~$\Ga$ 
such that $\hyp/\Ga$ is a Riemann surface of genus~$\tilde g=(g+1-n)/2$ with $n$~holes.
If $(a_1,b_1,\dots,a_{\tilde g},b_{\tilde g},c_1,\dots,c_n)$ is a standard generating set of~$\Ga$,
then
$$(a_1,b_1,\dots,a_{\tilde g},b_{\tilde g},c_1,\dots,c_n,\bar c_1,\dots,\bar c_n)$$
is a generating set of a real Fuchsian group~$\hGa$ of topological type $(g,k,1)$.
Any real Fuchsian group of topological type~$(g,k,1)$ is obtained this way.
\item
Let $(g,k,0)$ be a topological type of a Klein surface, i.e.\ $0\le k\le g$.
We choose $n\in\{k+1,\dots,g+1\}$ such that $n\equiv g+1~(\mod2)$.
Consider a Fuchsian group~$\Ga$
such that $\hyp/\Ga$ is a Riemann surface of genus~$\tilde g=(g+1-n)/2$ with $n$~holes.
If $(a_1,b_1,\dots,a_{\tilde g},b_{\tilde g},c_1,\dots,c_n)$ is a standard generating set of~$\Ga$,
then
$$(a_1,b_1,\dots,a_{\tilde g},b_{\tilde g},c_1,\dots,c_n,\bar c_1,\dots,\bar c_k,\tilde c_{k+1},\cdots,\tilde c_n)$$
is a generating set of a real Fuchsian group~$\hGa$ of topological type $(g,k,0)$.
Any real Fuchsian group of topological type~$(g,k,0)$ is obtained this way.
\item
Let $\hGa$ be a real Fuchsian group as in part~1 or~2 and $(P,\tau)$ be the corresponding Klein surface.
We now interpret the elements $(a_1,b_1,\dots,a_{\tilde g},b_{\tilde g},c_1,\dots,c_n)$ in~$\hGa$
as loops in $P$ without a base point up to homotopy of free loops.
We have $P^{\tau}=c_1\cup\dots\cup c_k$.
The curves $c_1,\dots,c_k$ correspond to ovals, the curves $c_{k+1},\dots,c_n$ correspond to twists.
Let $P_1$ and~$P_2$ be the connected components of the complement of the curves $c_1,\dots,c_n$ in~$P$.
Each of these components is a surface of genus $\tilde g=(g+1-n)/2$ with $n$~holes.
We have $\tau(P_1)=P_2$.
We will refer to $P_1$ and $P_2$ as {\it decomposition of}~$(P,\tau)$ {\it in two halves}.
(Note that such a decomposition is unique if $(P,\tau)$ is separating,
but is not unique if $(P,\tau)$ is non-separating since the twists $c_{k+1},\dots,c_n$ can be chosen in different ways.)
Then
$$(a_1,b_1,\dots,a_{\tilde g},b_{\tilde g},c_1,\dots,c_n)$$
is a generating set of $\piorb(P_1)$,
while its image under~$\tau$ gives a generating set of $\piorb(P_2)$.
\end{enumerate}
\end{thm}

\begin{figure}[H]
\begin{center}
\leavevmode
  \setcoordinatesystem units <1cm,1cm> point at 0 0
  \setplotarea x from -8 to 8, y from -4 to 3
  \plot -6 -4 6 -4 6 3 -6 3 -6 -4 /
  \plot -2.5 -4 -2.5 3 /
  \plot 1.5 -4 1.5 3 /
  \ellipticalarc axes ratio 1:2 360 degrees from -4.25 1 center at -4.25 1.5
  \arrow <10pt> [0.2,0.5] from -4.7 1.7 to -4.7 1.3
  \ellipticalarc axes ratio 1:2 360 degrees from -4.25 -1 center at -4.25 -1.5
  \arrow <10pt> [0.2,0.5] from -4.7 -1.3 to -4.7 -1.7
  \arrow <10pt> [0.2,0.5] from -5.15 -0.2 to -5.15 0.2
  \arrow <10pt> [0.2,0.5] from -3.35 -0.2 to -3.35 0.2
  \put {$c_i$} [r] <-5pt,0pt> at -4.7 1.5
  \put {$c_j$} [r] <-5pt,0pt> at -4.7 -1.5
  \put {$r_i$} [l] <10pt,0pt> at -4.25 1
  \put {$r_j$} [l] <10pt,0pt> at -4.25 -1
  \put {$\ell$} [r] <-5pt,0pt> at -5.15 0
  \put {$\tau\ell$} [l] <5pt,0pt> at -3.35 0
  \put {$c_i,c_j$~\text{ovals}} [c] at -4.25 -3
  \ellipticalarc axes ratio 1:2 360 degrees from -1 1 center at -1 1.5
  \arrow <10pt> [0.2,0.5] from -1.45 1.7 to -1.45 1.3
  \ellipticalarc axes ratio 1:2 360 degrees from -1 -1 center at -1 -1.5
  \arrow <10pt> [0.2,0.5] from -1.45 -1.3 to -1.45 -1.7
  \arrow <10pt> [0.2,0.5] from -1.9 -0.2 to -1.9 0.2
  \arrow <10pt> [0.2,0.5] from 0.8 -0.7 to 0.8 -0.3
  \arrow <10pt> [0.2,0.5] from -0.55 -1.7 to -0.55 -1.3
  \put {$c_i$} [r] <-5pt,0pt> at -1.45 1.5
  \put {$c_j$} [r] <-5pt,0pt> at -1.45 -1.5
  \put {$r_i$} [l] <10pt,4pt> at -1 1
  \put {$r_j$} [l] <5pt,0pt> at -0.55 -1.5
  \put {$\ell$} [r] <-5pt,0pt> at -1.9 0
  \put {$\tau\ell$} [l] <5pt,0pt> at 0.8 -0.5
  \put {$c_i$~\text{oval}, $c_j$~\text{twist}} [c] at -0.5 -3
  \ellipticalarc axes ratio 1:2 360 degrees from 3 1 center at 3 1.5
  \ellipticalarc axes ratio 1:2 360 degrees from 3 -1 center at 3 -1.5
  \arrow <10pt> [0.2,0.5] from 2.55 1.7 to 2.55 1.3
  \arrow <10pt> [0.2,0.5] from 2.55 -1.3 to 2.55 -1.7
  \arrow <10pt> [0.2,0.5] from 2.1 -0.2 to 2.1 0.2
  \arrow <10pt> [0.2,0.5] from 5.4 -0.2 to 5.4 0.2
  \arrow <10pt> [0.2,0.5] from 3.45 -1.7 to 3.45 -1.3
  \arrow <10pt> [0.2,0.5] from 3.45 1.3 to 3.45 1.7
  \put {$c_i$} [r] <-5pt,0pt> at 2.45 1.5
  \put {$c_j$} [r] <-5pt,0pt> at 2.45 -1.5
  \put {$\ell$} [r] <-5pt,0pt> at 2.1 0
  \put {$\tau\ell$} [l] <5pt,0pt> at 5.4 0
  \put {$r_i$} [l] <5pt,0pt> at 3.45 1.5
  \put {$r_j$} [l] <5pt,0pt> at 3.45 -1.5
  \put {$c_i,c_j$~\text{twists}} [c] at 3.75 -3
  \setplotsymbol({$\bullet$})
  \ellipticalarc axes ratio 1:2 360 degrees from -4.25 1 center at -4.25 0
  \arrow <10pt> [0.3,0.8] from -4.75 -0.3 to -4.75 0.3
  \arrow <10pt> [0.3,0.8] from -3.75 0.3 to -3.75 -0.3
  \ellipticalarc axes ratio 1:2 180 degrees from -1 1 center at -1 0
  \ellipticalarc axes ratio 1:1 180 degrees from -1 -2 center at -1 -0.5
  \ellipticalarc axes ratio 1:2 -180 degrees from -1 -1 center at -1 -1.5
  \arrow <10pt> [0.3,0.8] from -1.5 -0.2 to -1.5 0.2
  \arrow <10pt> [0.3,0.8] from 0.5 -0.3 to 0.5 -0.7
  \ellipticalarc axes ratio 1:2 180 degrees from 3 1 center at 3 0
  \ellipticalarc axes ratio 1:1 180 degrees from 3 -2 center at 3 0
  \ellipticalarc axes ratio 1:2 180 degrees from 3 1 center at 3 1.5
  \ellipticalarc axes ratio 1:2 -180 degrees from 3 -1 center at 3 -1.5
  \arrow <10pt> [0.3,0.8] from 2.5 -0.2 to 2.5 0.2
  \arrow <10pt> [0.3,0.8] from 5 0.2 to 5 -0.2
\end{center}
\caption{Bridges}
\label{fig-bridges}
\end{figure}

\begin{mydef}
\label{def-bridge}
Let $P_1$ and $P_2$ be a decomposition of a Klein surface~$(P,\tau)$ in two halves as in Theorem~\ref{thm-gen}.
For two invariant closed curves~$c_i$ and~$c_j$, a {\it bridge} between~$c_i$ and~$c_j$ is a curve of the form
$$r_i\cup(\tau\ell)^{-1}\cup r_j\cup\ell,$$
where:
\begin{enumerate}[$\bullet$]
\item
$\ell$ is a simple path in~$P_1$ starting on~$c_j$ and ending on~$c_i$.
\item
$r_i$ is the path along~$c_i$ from the end point of~$\ell$ to the end point of~$\tau\ell$.
(If $c_i$ is an oval then the path $r_i$ consists of one point.)
\item
$r_j$ is the path along~$c_j$ from the starting point of~$\tau\ell$ to the starting point of~$\ell$.
(If $c_j$ is an oval then the path $r_j$ consists of one point.)
\end{enumerate}
Figure~\ref{fig-bridges} shows the shapes of the bridges for different types of invariant curves.
The bridges are shown in bold.
The bold arrows on the bold lines show the direction of the bridges,
while the thinner arrows near the lines show the directions of the paths $c_i$, $c_j$, $r_i$, $r_j$, $\ell$ and $\tau\ell$. 
\end{mydef}

\begin{mydef}
\label{def-symm-gen}
Let $(P,\tau)$ be a Klein surface with a decomposition in two halves $P_1$ and $P_2$ as in Theorem~\ref{thm-gen}.
A {\it symmetric generating set}
of~$\piorb(P)$ is a generating set of the form
$$
  (
   a_1,b_1,\dots,a_{\tilde g},b_{\tilde g},
   a_1',b_1',\dots,a_{\tilde g}',b_{\tilde g}',
   c_1,\dots,c_{n-1},d_1,\dots,d_{n-1}
  ),
$$
where
\begin{enumerate}[$\bullet$]
\item
$(a_1,b_1,\dots,a_{\tilde g},b_{\tilde g},c_1,\dots,c_n)$ is a generating set of $\piorb(P_1)$ as in Theorem~\ref{thm-gen},
\item
$a_i'=(\tau a_i)^{-1}$ and $b_i'=(\tau b_i)^{-1}$ for $i=1,\dots,\tilde g$.
\item
$d_1,\dots,d_{n-1}$ are closed curves which only intersect at the base point,
such that $d_i$ is homotopic to a bridge between $c_i$ and $c_n$,
\end{enumerate}
Note that $\tau c_i=c_i$ and $\tau d_i=c_i^{|c_i|}d_i^{-1}c_n^{|c_n|}$,
where $|c_j|=0$ if $c_j$ is an oval and $|c_j|=1$ if $c_j$ is a twist.
\end{mydef}

\begin{rem}
\label{standard-symmetric}
Note that a symmetric generating set is not a standard generating set in the sense of Definition~\ref{def-standard-basis},
however it is free homotopic to a standard one.
\end{rem}

\subsection{Topological Invariants of Higher Arf Functions}

\label{subsec-topinv-arf}

\myskip
In this section we summarize the results of~\cite{NP:2005, NP:2009} on topological invariants of higher Arf functions.

\begin{mydef}
\label{def-arf-inv}
Let $\arf:\piorbO(P)\to\z/m\z$ be an $m$-Arf function.
For $g=1$ we define the {\it Arf invariant} $\de=\de(P,\arf)$ as
$$\de=\gcd(m,\arf(a_1),\arf(b_1),\arf(c_1)+1,\dots,\arf(c_n)+1),$$
where
$$\{a_1,b_1,c_i~(i=1,\dots,n)\}$$
is a standard (or symmetric) generating set of the fundamental group $\piorb(P)$.
For $g\ge2$ and even $m$ we define the {\it Arf invariant} $\de=\de(P,\arf)$ as $\de=0$
if there is a standard (or symmetric) generating set
$$\{a_i,b_i~(i=1,\dots,g),c_i~(i=1,\dots,n)\}$$
of the fundamental group $\piorb(P)$ such that
$$\sum\limits_{i=1}^g(1-\arf(a_i))(1-\arf(b_i))\equiv0~(\mod2)$$
and as $\de=1$ otherwise.
For~$g\ge2$ and odd $m$ we set $\de=0$.
\end{mydef}

\begin{mydef}
For even~$m$ and $g\ge2$ we say that an $m$-Arf function with the Arf invariant~$\de$ is {\it even} if $\de=0$ and {\it odd} if $\de=1$.
\end{mydef}

\begin{rem}
For even~$m$, $g\ge2$ and $\de=0$, 
it is not true that the expression
$$\sum\limits_{i=1}^g(1-\arf(a_i))(1-\arf(b_i))$$
is even for any standard generating set.
This expression only needs to be even for one generating set
and can in fact have different parity for different generating sets.
On the other hand we will see that the parity of this expression
does not depend on the choice of the generating set
if $\arf(c_1),\dots,\arf(c_n)$ are all odd (Theorem~\ref{thm-arf-construction})
or if $P$ is compact (Proposition~\ref{prop-arf-compact}).
\end{rem}

\begin{rem}
The Arf invariant~$\de$ is a topological invariant of an Arf function~$\arf$,
i.e.\ it does not change under self-homeomorphisms of the Riemann surface~$P$.
\end{rem}


The following is a special case of our earlier classification result, Theorem~5.3 in~\cite{NP:2009}:

\begin{thm}
\label{thm-arf-existence}
Let $P$ be a hyperbolic Riemann surface of genus~$g$ with $n$~holes.
Let $c_1,\dots,c_n$ be closed curves around the holes as in Definition~\ref{def-standard-basis}.
Let $\arf$ be an $m$-Arf function on~$P$ and let $\de$ be the $m$-Arf invariant of~$\arf$.
Then 
\begin{enumerate}[(a)]
\item
If $g\ge2$ and $m\equiv1~(\mod2)$ then $\de=0$.
\item
If $g\ge2$ and $m\equiv0~(\mod2)$ and $\arf(c_i)\equiv0~(\mod2)$ for some $i$ then $\de=0$.
\item
If $g=1$ then $\de$ is a divisor of~$\gcd(m,\arf(c_1)+1,\dots,\arf(c_n)+1)$.
\item
$\arf(c_1)+\cdots+\arf(c_n)\equiv(2-2g)-n~(\mod m)$.
\end{enumerate}
\end{thm}


The following result describes the construction of such $m$-Arf functions.
It follows from Lemma ~3.7, Lemma 3.9, Theorem~4.9 and the proof of Theorem~5.3 in~\cite{NP:2009}.

\begin{thm}
\label{thm-arf-construction}
Let $P$ be a hyperbolic Riemann surface of genus~$g$ with $n$~holes. 
Then for any standard generating set
$$(a_1,b_1,\dots,a_g,b_g,c_1,\dots,c_n)$$
of $\piorb(P)$ and any choice of values $\al_1,\be_1,\dots,\al_g,\be_g,\ga_1,\dots,\ga_n$ in $\z/m\z$ with 
$$\ga_1+\cdots+\ga_n\equiv(2-2g)-n~(\mod m),$$
there exists an $m$-Arf function~$\arf$ on~$P$
such that $\arf(a_i)=\al_i$, $\arf(b_i)=\be_i$ for $i=1,\dots,g$ and $\arf(c_i)=\ga_i$ if $i=1,\dots,n$. 
The Arf invariant~$\de$ of this $m$-Arf function~$\arf$ satisfies the following conditions:
\begin{enumerate}[(a)]
\item
If $g\ge2$ and $m\equiv1~(\mod2)$ then $\de=0$.
\item
If $g\ge2$ and $m\equiv0~(\mod 2)$ and $\ga_i\equiv0~(\mod2)$ for some $i$ then $\de=0$.
\item
\label{part-c}
If $g\ge2$ and $m\equiv0~(\mod2)$ and $\ga_1\equiv\cdots\equiv\ga_n\equiv1~(\mod2)$ then $\de\in\{0,1\}$ and
$$\de\equiv\sum\limits_{i=1}^g\,(1-\al_i)(1-\be_i)~(\mod2).$$
\item
If $g=1$ then $\de=\gcd(m,\al_1,\be_1,\ga_1+1,\dots,\ga_n+1)$.
\end{enumerate}
\end{thm}



In the special case of a compact Riemann surface without holes (i.e.\ $n=0$) we have

\begin{prop}
\label{prop-arf-compact}
Let $P$ be a compact Riemann surface of genus~$g\ge2$. 
Assume that $2-2g\equiv0~(\mod m)$.
Then for any standard generating set $(a_1,b_1,\dots,a_g,b_g)$ of $\piorb(P)$
and any choice of values $\al_1,\be_1,\dots,\al_g,\be_g$ in $\z/m\z$,
there exists an $m$-Arf function~$\arf$ on~$P$ such that $\arf(a_i)=\al_i$, $\arf(b_i)=\be_i$ for $i=1,\dots,g$.
The Arf invariant~$\de$ of this $m$-Arf function~$\arf$ satisfies the following conditions:
\begin{enumerate}[(a)]
\item
If $m$ is odd then $\de=0$.
\item
If $m$ is even then $\de\in\{0,1\}$ and
$$\de\equiv\sum\limits_{i=1}^g\,(1-\al_i)(1-\be_i)~(\mod2).$$
\end{enumerate}
\end{prop}

\subsection{Values of Real Arf Functions on Symmetric Generating Sets}


\label{subsec-checking-real}

\begin{lem}
\label{lem-bridges}
Let $(P,\tau)$ be a Klein surface and let $\arf$ be an $m$-Arf function on~$P$.
(Here we do not assume that $\arf$ is a real Arf function.)
Let $d$ be a bridge (Definition~\ref{def-bridge}).
\begin{enumerate}[$\bullet$]
\item
Let $(P,\tau)$ be separating. 
Then 
$$\arf(\tau d)=-\arf(d).$$
\item
Let $(P,\tau)$ be non-separating.
Let $c_1,\dots,c_n$ be invariant curves
such that the first $k$ correspond to ovals and the next $n-k$ correspond to twists
(see Theorem~\ref{thm-gen}).
Assume that $\arf$ vanishes on all twists $c_{k+1},\dots,c_n$.
Then
$$\arf(\tau d)=-\arf(d).$$
\end{enumerate}
\end{lem}

\begin{proof}
Let $d=r_i\cup(\tau\ell)^{-1}\cup r_j\cup\ell$ be a bridge between~$c_i$ and~$c_j$.
\begin{enumerate}[$\bullet$]
\item 
If $c_i$ and $c_j$ are both ovals we have (see Figure~\ref{fig-bridge-oval-oval}):
\begin{align*}
  d&=(\tau\ell)^{-1}\cup\ell,\\
  \tau d&=\ell^{-1}\cup\tau\ell=d^{-1}.
\end{align*}
Now we see that $\arf(\tau d)=\arf(d^{-1})=-\arf(d)$.

\begin{figure}[h]
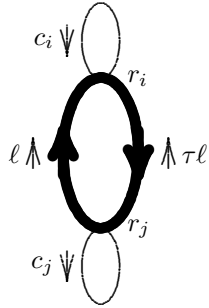

\begin{center}
\leavevmode
  \setcoordinatesystem units <1cm,1cm> point at 0 0
  \setplotarea x from -2 to 2, y from -2 to 2
  \ellipticalarc axes ratio 1:2 360 degrees from 0 1 center at 0 1.5
  \arrow <10pt> [0.2,0.5] from -0.45 1.7 to -0.45 1.3
  \ellipticalarc axes ratio 1:2 360 degrees from 0 -1 center at 0 -1.5
  \arrow <10pt> [0.2,0.5] from -0.45 -1.3 to -0.45 -1.7
  \arrow <10pt> [0.2,0.5] from -0.9 -0.2 to -0.9 0.2
  \arrow <10pt> [0.2,0.5] from 0.9 -0.2 to 0.9 0.2
  \setplotsymbol({$\bullet$})
  \ellipticalarc axes ratio 1:2 360 degrees from 0 1 center at 0 0
  \arrow <10pt> [0.3,0.8] from -0.5 -0.3 to -0.5 0.3
  \arrow <10pt> [0.3,0.8] from 0.5 0.3 to 0.5 -0.3
  \put {$c_i$} [r] <-5pt,0pt> at -0.45 1.5
  \put {$c_j$} [r] <-5pt,0pt> at -0.45 -1.5
  \put {$r_i$} [l] <10pt,0pt> at 0 1
  \put {$r_j$} [l] <10pt,0pt> at 0 -1
  \put {$\ell$} [r] <-5pt,0pt> at -0.9 0
  \put {$\tau\ell$} [l] <5pt,0pt> at 0.9 0
  \multiput {$\bullet$} at 0 1 /
\end{center}
\caption{A bridge between two ovals $c_i$ and $c_j$}
\label{fig-bridge-oval-oval}
\end{figure}

\item 
If $c_i$ is an oval and $c_j$ is a twist we have (see Figure~\ref{fig-bridge-oval-twist}):
\begin{align*}
  d&=(\tau\ell)^{-1}\cup r_j\cup\ell,\\
  \tau d&=\ell^{-1}\cup\tau r_j\cup\tau\ell,\\
  (\tau d)^{-1}
  &=(\tau\ell)^{-1}\cup(\tau r_j)^{-1}\cup\ell\\
  &=\big((\tau\ell)^{-1}\cup r_j\cup\ell\big)\cup\big(\ell^{-1}\cup r_j^{-1}\cup(\tau r_j)^{-1}\cup\ell\big)\\
  &=d\cup c_j^{-1}.
\end{align*}
Using Property~\ref{arf-prop-cross} of Arf functions we obtain
$$\arf((\tau d)^{-1})=\arf(d)+\arf(c_j^{-1}).$$
Using Property~\ref{arf-prop-inv} of Arf functions we obtain
$$-\arf(\tau d)=\arf(d)-\arf(c_j).$$
Now we see that $\arf(c_j)=0$ implies $\arf(\tau d)=-\arf(d)$.

\begin{figure}[h]
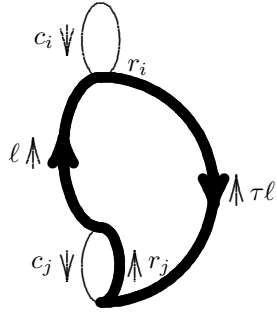

\begin{center}
\leavevmode
  \setcoordinatesystem units <1cm,1cm> point at 0 0
  \setplotarea x from -2 to 2, y from -2 to 2
  \ellipticalarc axes ratio 1:2 360 degrees from 0 1 center at 0 1.5
  \arrow <10pt> [0.2,0.5] from -0.45 1.7 to -0.45 1.3
  \ellipticalarc axes ratio 1:2 360 degrees from 0 -1 center at 0 -1.5
  \arrow <10pt> [0.2,0.5] from -0.45 -1.3 to -0.45 -1.7
  \arrow <10pt> [0.2,0.5] from -0.9 -0.2 to -0.9 0.2
  \arrow <10pt> [0.2,0.5] from 1.8 -0.7 to 1.8 -0.3
  \arrow <10pt> [0.2,0.5] from 0.45 -1.7 to 0.45 -1.3
  \setplotsymbol({$\bullet$})
  \ellipticalarc axes ratio 1:2 180 degrees from 0 1 center at 0 0
  \ellipticalarc axes ratio 1:1 180 degrees from 0 -2 center at 0 -0.5
  \ellipticalarc axes ratio 1:2 -180 degrees from 0 -1 center at 0 -1.5
  \arrow <10pt> [0.3,0.8] from -0.5 -0.2 to -0.5 0.2
  \arrow <10pt> [0.3,0.8] from 1.5 -0.3 to 1.5 -0.7
  \put {$c_i$} [r] <-5pt,0pt> at -0.45 1.5
  \put {$c_j$} [r] <-5pt,0pt> at -0.45 -1.5
  \put {$r_i$} [l] <10pt,4pt> at 0 1
  \put {$r_j$} [l] <5pt,0pt> at 0.45 -1.5
  \put {$\ell$} [r] <-5pt,0pt> at -0.9 0
  \put {$\tau\ell$} [l] <5pt,0pt> at 1.8 -0.5
  \multiput {$\bullet$} at 0 1 /
\end{center}
\caption{A bridge between an oval $c_i$ and a twist $c_j$}
\label{fig-bridge-oval-twist}
\end{figure}

\item 
If $c_i$ and $c_j$ are both twists we have (see Figure~\ref{fig-bridge-twist-twist}):
\begin{align*}
  d&=r_i\cup(\tau\ell)^{-1}\cup r_j\cup\ell,\\
  \tau d&=\ell^{-1}\cup\tau r_j\cup\tau\ell\cup\tau r_i
\end{align*}
and
\begin{align*}
  &(\tau d)^{-1}\\
  &=(\tau r_i)^{-1}\cup(\tau\ell)^{-1}\cup(\tau r_j)^{-1}\cup\ell\\
  &=((\tau r_i)^{-1}\cup r_i^{-1})
    \cup(r_i\cup(\tau\ell)^{-1}\cup r_j\cup\ell)
    \cup(\ell^{-1}\cup r_j^{-1}\cup(\tau r_j)^{-1}\cup\ell)\\
  &=c_i^{-1}\cup d\cup c_j^{-1}.
\end{align*}
Using Property~\ref{arf-prop-cross} of Arf functions we obtain
$$\arf((\tau d)^{-1})=\arf(c_i^{-1})+\arf(d)+\arf(c_j^{-1}).$$
Using Property~\ref{arf-prop-inv} of Arf functions we obtain
$$-\arf(\tau d)=\arf(d)-\arf(c_i)-\arf(c_j).$$
Now we see that $\arf(c_i)=\arf(c_j)=0$ implies $\arf(\tau d)=-\arf(d)$.

\begin{figure}[h]
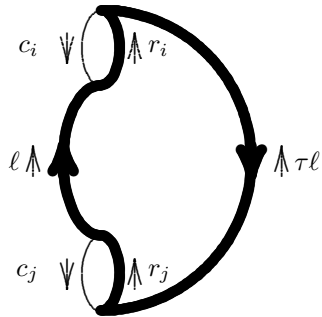

\begin{center}
\leavevmode
  \setcoordinatesystem units <1cm,1cm> point at 0 0
  \setplotarea x from -2 to 2, y from -2 to 2
  \ellipticalarc axes ratio 1:2 360 degrees from 0 1 center at 0 1.5
  \ellipticalarc axes ratio 1:2 360 degrees from 0 -1 center at 0 -1.5
  \arrow <10pt> [0.2,0.5] from -0.45 1.7 to -0.45 1.3
  \arrow <10pt> [0.2,0.5] from -0.45 -1.3 to -0.45 -1.7
  \arrow <10pt> [0.2,0.5] from -0.9 -0.2 to -0.9 0.2
  \arrow <10pt> [0.2,0.5] from 2.4 -0.2 to 2.4 0.2
  \arrow <10pt> [0.2,0.5] from 0.45 -1.7 to 0.45 -1.3
  \arrow <10pt> [0.2,0.5] from 0.45 1.3 to 0.45 1.7
  \setplotsymbol({$\bullet$})
  \ellipticalarc axes ratio 1:2 180 degrees from 0 1 center at 0 0
  \ellipticalarc axes ratio 1:1 180 degrees from 0 -2 center at 0 0
  \ellipticalarc axes ratio 1:2 180 degrees from 0 1 center at 0 1.5
  \ellipticalarc axes ratio 1:2 -180 degrees from 0 -1 center at 0 -1.5
  \arrow <10pt> [0.3,0.8] from -0.5 -0.2 to -0.5 0.2
  \arrow <10pt> [0.3,0.8] from 2 0.2 to 2 -0.2
  \put {$c_i$} [r] <-5pt,0pt> at -0.65 1.5
  \put {$c_j$} [r] <-5pt,0pt> at -0.65 -1.5
  \put {$\ell$} [r] <-5pt,0pt> at -0.9 0
  \put {$\tau\ell$} [l] <5pt,0pt> at 2.4 0
  \put {$r_i$} [l] <5pt,0pt> at 0.45 1.5
  \put {$r_j$} [l] <5pt,0pt> at 0.45 -1.5
  \multiput {$\bullet$} at 0 1 /
\end{center}
\caption{A bridge between two twists $c_i$ and $c_j$}
\label{fig-bridge-twist-twist}
\end{figure}
\end{enumerate}
\end{proof}

\begin{lem}
\label{lem-checking-real}
Let $(P,\tau)$ be a Klein surface of type $(g,k,\ve)$ and $\arf$ an $m$-Arf function on~$P$.
Let $c_1,\dots,c_n$ be invariant curves as in Theorem~\ref{thm-gen},
with $c_1,\dots,c_k$ corresponding to ovals and $c_{k+1},\dots,c_n$ corresponding to twists.
Let
$$
  \calB
  =
  (
    a_1,b_1,\dots,a_{\tilde g},b_{\tilde g},a_1',b_1',\dots,a_{\tilde g}',b_{\tilde g}',
    c_1,\dots,c_{n-1},d_1,\dots,d_{n-1}
  )
$$
be a symmetric generating set of~$\piorb(P)$ (Definition~\ref{def-symm-gen}).
Assume that
\begin{enumerate}[(i)]
\item
$\arf(a_i)\equiv\arf(a_i')~(\mod m)$, $\arf(b_i)\equiv\arf(b_i')~(\mod m)$ for $i=1,\dots,\tilde g$,
\item
$2\arf(c_i)\equiv0~(\mod m)$ for $i=1,\dots,n-1$. 
\item
$\arf(c_i)\equiv0~(\mod m)$ for $i=k+1,\dots,n$.
\end{enumerate}
Then $\arf$ is a real $m$-Arf function on $(P,\tau)$.
\end{lem}

\begin{rem}
Condition~(iii) means that $\arf$ vanishes on all twists.
\end{rem}

\begin{proof}
To be real, the $m$-Arf function~$\arf$ must vanish on all twists
and be compatible with~$\tau$, i.e.\ satisfy the equation $\arf(\tau x)\equiv-\arf(x)~(\mod m)$ for all~$x\in\piorbO(P)$.
Condition~(iii) implies that $\arf$ vanishes on all twists.
We will first check the equation $\arf(\tau x)\equiv-\arf(x)~(\mod m)$ for all~$x$ in~$\calB$.
\begin{enumerate}[$\bullet$]
\item
$x=a_i,b_i$, $i=1,\dots,\tilde g$:
Recall that $a_i'=(\tau a_i)^{-1}$,
hence $\tau a_i=(a_i')^{-1}$ and $\arf(\tau a_i)\equiv\arf((a_i)'^{-1})\equiv-\arf(a_i')$.
Condition~(i) implies $\arf(a_i')\equiv\arf(a_i)$, hence $\arf(\tau a_i)\equiv-\arf(a_i')\equiv-\arf(a_i)$.
Similarly $\arf(\tau b_i)\equiv-\arf(b_i)$.
\item
$x=a_i',b_i'$, $i=1,\dots,\tilde g$:
Recall that $a_i'=(\tau a_i)^{-1}$,
hence $\tau a_i'=a_i^{-1}$ and $\arf(\tau a_i')\equiv\arf(a_i^{-1})\equiv-\arf(a_i)$.
Condition~(i) implies $\arf(a_i)\equiv\arf(a_i')$, hence $\arf(\tau a_i')\equiv-\arf(a_i)\equiv-\arf(a_i')$.
Similarly $\arf(\tau b_i')\equiv-\arf(b_i')$.
\item
$x=c_i$, $i=1,\dots,n-1$:
Recall that $\tau c_i=c_i$ for an oval~$c_i$, $i=1,\dots,k$,
while $\tau c_i$ is conjugate to $c_i$ for a twist~$c_i$, $i=k+1,\dots,n-1$.
In both cases $\arf(\tau c_i)\equiv\arf(c_i)$.
Condition~(ii) implies $2\arf(c_i)\equiv0$ for $i=1,\dots,n-1$, hence $\arf(\tau c_i)\equiv\arf(c_i)\equiv-\arf(c_i)$.
\item
$x=d_i$, $i=1,\dots,n-1$:
Condition~(iii) implies that $\arf$ vanishes on all twists.
According to Lemma~\ref{lem-bridges}, it follows that $\arf(\tau d_i)\equiv-\arf(d_i)$ for $i=1,\dots,n-1$.
\end{enumerate}
Consider $\tilde\arf:\piorb(P)\to\z/m\z$ given by $\tilde\arf(x)=-\arf(\tau x)$.
The involution~$\tau$ is an orientation-reversing homeomorphism,
so it is easy to check using Definition~\ref{def-m-arf} that if $\arf$ is an $m$-Arf function then so is $\tilde\arf$.
We have checked that for any $x$ in $\calB$ we have $\arf(\tau x)\equiv-\arf(x)$,
hence $\tilde\arf(x)\equiv-\arf(\tau x)\equiv\arf(x)$.
Thus we have two $m$-Arf functions, $\arf$ and $\tilde\arf$, which coincide on a generating set.
We can conclude that $\arf$ and $\tilde\arf$ coincide everywhere on $\piorb(P)$,
i.e.\ for any~$x\in\piorb(P)$ we have $\arf(\tau x)\equiv-\tilde\arf(x)\equiv-\arf(x)$.
This shows that the Arf function~$\arf$ is real.
\end{proof}

\subsection{Classification and Enumeration of Real Arf Functions}

\label{subsec-classification-enum}

\myskip
This section contains the main results of the paper.

\myskip
Let $(P,\tau)$ be a Klein surface of type $(g,k,\ve)$, $g\ge2$.
Let $c_1,\dots,c_n$ be invariant curves as in Theorem~\ref{thm-gen}.
The curves $c_1,\dots,c_k$ correspond to ovals.
In the separating case ($\ve=1$) we have $n=k$.
In the non-separating case ($\ve=0$) we have $n>k$ and the curves $c_{k+1},\dots,c_n$ correspond to twists.
Let
$$
  \calB
  =
  (
    a_1,b_1,\dots,a_{\tilde g},b_{\tilde g},a_1',b_1',\dots,a_{\tilde g}',b_{\tilde g}',
    c_1,\dots,c_{n-1},d_1,\dots,d_{n-1}
  )
$$
be a symmetric generating set of~$\piorb(P)$ (Definition~\ref{def-symm-gen}).
Let $P_1$ and~$P_2$ be the connected components of the complement of the curves~$c_1,\dots,c_n$ in~$P$.
Each of these components is a surface of genus $\tilde g=(g+1-n)/2$ with $n$~holes.
We have $\tau(P_1)=P_2$.

\begin{lem}
\label{lem-half-sum}
Let $\arf$ be a real $m$-Arf function on~$(P,\tau)$, then
$$\arf(c_1)+\cdots+\arf(c_n)\equiv1-g~(\mod m)$$
and
\begin{align*}
   &g\equiv1~(\mod\frac{m}{2})\quad\text{if}~m~\text{is even},\\
    &g\equiv1~(\mod m)\quad\text{if}~m~\text{is odd}.
\end{align*}
\end{lem}

\begin{proof}
Theorem~\ref{thm-arf-existence}, applied to~$P_1$, implies that
$$\arf(c_1)+\cdots+\arf(c_n)\equiv(2-2\tilde g)-n\equiv1-g~(\mod m).$$
Moreover, $\arf(c_i)\equiv0~(\mod(m/2))$ for even~$m$ and $\arf(c_i)\equiv0~(\mod m)$ for odd~$m$ completes the proof.
\end{proof}


Our main result for even~$m$ is as follows:

\begin{thm}
\label{class-even}
Let $m$ be even.
\begin{enumerate}[1)]
\item
Let $\arf$ be a real $m$-Arf function on~$(P,\tau)$.
Then
\begin{align*}
  &\arf(a_i)\equiv\arf(a_i')~(\mod m)~\text{and}~\arf(b_i)\equiv\arf(b_i')~(\mod m)\quad\text{for}~i=1,\dots,\tilde g,\\
  &\arf(c_i)\equiv0~(\mod(m/2))\quad\text{for}~i=1,\dots,k,\\
  &\arf(c_1)+\cdots+\arf(c_k)\equiv1-g~(\mod m),\\
  &g\equiv1~(\mod(m/2)),\\
  &\text{and for}~\ve=0~\text{additionally}~\arf(c_i)\equiv0~(\mod m)\quad\text{for}~i=k+1,\dots,n.
\end{align*}
\item
Let a value set~$\calV$ in $(\z/m\z)^{4\tilde g+2n-2}$ be
$$
  (
   \al_1,\be_1,\dots,\al_{\tilde g},\be_{\tilde g},\al_1',\be_1',\dots,\al_{\tilde g}',\be_{\tilde g}',
   \ga_1,\dots,\ga_{n-1},\de_1,\dots,\de_{n-1}
  ).
$$
Assume that
\begin{align*}
  &\al_i=\al_i'~\text{and}~\be_i=\be_i'\quad\text{for}~i=1,\dots,\tilde g,\\
  &\ga_1,\dots,\ga_{k-1}\in\{0,m/2\}.
\end{align*}
In the case $\ve=0$ assume also that
\begin{align*}
  &\ga_k\in\{0,m/2\},\\
  &\ga_1+\cdots+\ga_k\equiv1-g~(\mod m),\\
  &\ga_{k+1}=\cdots=\ga_{n-1}=0.
\end{align*}
In the case $\ve=1$ assume also that
$$g\equiv1~(\mod\frac{m}{2}).$$
Then there exists a real $m$-Arf function $\arf$ on $(P,\tau)$ with
$$\arf(a_i)=\al_i,~\arf(b_i)=\be_i,~\arf(a_i')=\al_i',~\arf(b_i')=\be_i',~\arf(c_i)=\ga_i,~\arf(d_i)=\de_i.$$
For this $m$-Arf function we have
\begin{align*}
  &\arf(c_n)\equiv0~(\mod m)~\text{in the case}~\ve=0,\\
  &\arf(c_n)\equiv(1-g)-(\ga_1+\cdots+\ga_{n-1})~(\mod m)~\text{in the case}~\ve=1.
\end{align*}
\item
The number of real $m$-Arf functions on~$(P,\tau)$ is
\begin{align*}
  &m^g\quad\text{in the case}~\ve=0, k=0,\\
  &m^g\cdot 2^{k-1}\quad\text{otherwise}.
\end{align*}
\item
The Arf invariant $\de\in\{0,1\}$ of a real $m$-Arf function~$\arf$ on~$(P,\tau)$ is given by
$$\de\equiv\sum\limits_{i=1}^{n-1}(1-\arf(c_i))(1-\arf(d_i))~(\mod2).$$
\item
Consider $\ga_1,\dots,\ga_{n-1}$ as above.
Let
$$\Si=\sum\limits_{i=1}^{n-1}(1-\ga_i)(1-\de_i).$$
In the case $\ve=1$, $m\equiv2~(\mod4)$, $\ga_1=\dots=\ga_{n-1}=m/2$,
any choice of $(\de_1,\dots,\de_{n-1})\in(\z/m\z)^{k-1}$ gives $\Si\equiv0~(\mod2)$.
In all other cases,
out of the $m^{n-1}$ possible choices for $(\de_1,\dots,\de_{n-1})\in(\z/m\z)^{n-1}$,
there are $m^{n-1}/2$ which give $\Si\equiv0~(\mod2)$ and $m^{n-1}/2$ which give $\Si\equiv1~(\mod2)$.
\item
In the case~$\ve=0$, the numbers of even and odd real $m$-Arf functions on $(P,\tau)$ are both equal to
$$\frac{m^g}{2}\quad\text{for}~k=0\quad\text{and}\quad m^g\cdot 2^{k-2}\quad\text{for}~k\ge1.$$
In the case~$\ve=1$ and $m\equiv0~(\mod4)$, the numbers of even and odd real $m$-Arf functions are both equal to
$$m^g\cdot 2^{k-2}.$$
In the case $\ve=1$ and $m\equiv2~(\mod4)$, the numbers of even and odd real $m$-Arf functions respectively are
$$m^g\cdot\frac{2^{k-1}+1}2\quad\text{and}\quad m^g\cdot\frac{2^{k-1}-1}2.$$
\end{enumerate}
\end{thm}

\begin{proof}
\begin{enumerate}[1)]
\item
For a real $m$-Arf function $\arf$ on~$(P,\tau)$ we have
$$\arf(a_i')\equiv\arf((\tau a_i)^{-1})\equiv-\arf(\tau a_i)\equiv\arf(a_i)$$
and similarly $\arf(b_i')\equiv\arf(b_i)$ by definition
and $\arf(c_i)\equiv0~(\mod(m/2))$ for $i=1,\dots,k$ according to Lemma~\ref{lem-compatible-arf-zero-or-half-m}.
Lemma~\ref{lem-half-sum} implies
$$\arf(c_1)+\cdots+\arf(c_n)\equiv1-g~(\mod m)$$
and $g\equiv1~(\mod(m/2))$.
In the case~$\ve=0$ we also have from Definition~\ref{def-real-arf}
that $\arf(c_i)\equiv0~(\mod m)$ for $i=k+1,\dots,n$ and
$$
  \arf(c_1)+\cdots+\arf(c_k)
  \equiv\arf(c_1)+\cdots+\arf(c_n)
  \equiv1-g~(\mod m).
$$
In the case~$\ve=1$ we have $n=k$ and hence
$$
  \arf(c_1)+\cdots+\arf(c_k)
  =\arf(c_1)+\cdots+\arf(c_n)
  \equiv(1-g)~(\mod m).
$$
\item
In the case~$\ve=0$, we know that $1-g\equiv\ga_1+\cdots+\ga_k~(\mod m)$
and $\ga_i\equiv0~(\mod(m/2))$ for $i=1,\dots,k$, hence $2-2g\equiv0~(\mod m)$.
In the case~$\ve=1$, we know that $1-g\equiv0~(\mod(m/2))$, hence $2-2g\equiv0~(\mod m)$.
This implies, according to Proposition~\ref{prop-arf-compact},
that the values~$\calV$ on~$\calB$ determine a unique $m$-Arf function~$\arf$ on~$P$.
According to Lemma~\ref{lem-checking-real}, to show that this $m$-Arf function~$\arf$ is real,
it is sufficient to show that $\arf(c_n)\equiv0~(\mod m)$ in the case~$\ve=0$
and that $2\arf(c_n)\equiv0~(\mod m)$ in the case~$\ve=1$.
Lemma~\ref{lem-half-sum} implies that
$$1-g\equiv\arf(c_1)+\cdots+\arf(c_n)\equiv\ga_1+\cdots+\ga_{n-1}+\arf(c_n)~(\mod m),$$
hence
$$\arf(c_n)\equiv(1-g)-(\ga_1+\cdots+\ga_{n-1})~(\mod m).$$
In the case~$\ve=0$ we know that $\ga_1+\cdots+\ga_k\equiv1-g~(\mod m)$ and $\ga_{k+1}=\cdots\ga_{n-1}=0$, hence 
$$
  \arf(c_n)
  \equiv(1-g)-(\ga_1+\cdots+\ga_{n-1})
  \equiv(1-g)-(1-g)
  \equiv0~(\mod m).
$$
In the case~$\ve=1$ we know that $1-g,\ga_1,\dots,\ga_n\equiv0~(\mod(m/2))$,
hence
$$
  \arf(c_n)
  \equiv(1-g)-(\ga_1+\cdots+\ga_{n-1})
  \equiv0~(\mod(m/2))
$$
and therefore $2\arf(c_n)=0~(\mod m)$.
Thus in both cases $\arf$ is a real $m$-Arf function.
\item
For $\al_i$, $\be_i$, $\de_i$ we can take any values in~$\z/m\z$, 
i.e.\ the number of ways to choose $(\al_1,\be_1,\dots,\al_{\tilde g},\be_{\tilde g},\de_1,\dots,\de_{n-1})$ is
$$m^{2\tilde g+n-1}=m^g.$$
We have $\al_i'=\al_i$ and $\be_i'=\be_i$,
i.e.\ the values $\al_i'$ and $\be_i'$ are determined by the values $\al_i$ and $\be_i$.
We have $\ga_{k+1}=\cdots=\ga_{n-1}=0$ (for $\ve=0$), i.e.\ the values~$\ga_i$ for $i=k-1,\dots,n-1$ are fixed.
It remains to choose $\ga_1,\dots,\ga_{k-1}\in\{0,m/2\}$.
\begin{enumerate}[$\bullet$]
\item
In the case~$k=0$ (which is only possible for $\ve=0$),
there are $m^g$ different choices of~$\calV$ and hence different real $m$-Arf functions.
\item
In the case~$k\ge1$, there are $2^{k-1}$ ways to choose $\ga_1,\dots,\ga_{k-1}$.
Thus there are
$$m^g\cdot 2^{k-1}$$
different choices of~$\calV$ and hence different real $m$-Arf functions.
\end{enumerate}
\item
According to Proposition~\ref{prop-arf-compact} the Arf invariant~$\de\in\{0,1\}$ of an Arf function $\arf$ with values $\calV$ on $\calB$ is
$$
  \de
  \equiv
  \sum\limits_{i=1}^{\tilde g}(1-\al_i)(1-\be_i)
  +\sum\limits_{i=1}^{\tilde g}(1-\al_i')(1-\be_i')
  +\sum\limits_{i=1}^{n-1}(1-\ga_i)(1-\de_i)~(\mod2).
$$
Using $\al_i=\al_i'$ and $\be_i=\be_i'$ we obtain
$$
  \sum\limits_{i=1}^{\tilde g}(1-\al_i)(1-\be_i)+\sum\limits_{i=1}^{\tilde g}(1-\al_i')(1-\be_i')
  \equiv2\sum\limits_{i=1}^{\tilde g}(1-\al_i)(1-\be_i)
  \equiv0~(\mod2),
$$
hence
$$\de\equiv\sum\limits_{i=1}^{n-1}(1-\ga_i)(1-\de_i)~(\mod2).$$
Note that $(1-\ga_i)(1-\de_i)\not\equiv0~(\mod2)$ if and only if $\ga_i$ and $\de_i$ are both even.
Hence
$$\de\equiv\big|\{i\in\{1,\dots,n-1\}\st\ga_i~\text{and}~\de_i~\text{are even}\}\big|~(\mod2).$$
\item
We need to determine, for given $\ga_1,\dots,\ga_{n-1}$,
how many of the $m^{n-1}$ ways to choose $\de_1,\dots,\de_{n-1}$ lead to
$$\Si=\sum\limits_{i=1}^{n-1}(1-\ga_i)(1-\de_i)$$
being even and odd respectively.
If there exists $r\in\{1,\dots,n-1\}$ such that $\ga_r$ is even,
then for any choice of $\de_1,\dots,\hat{\de_r},\dots,\de_{n-1}$ 
exactly half of the possible choices for~$\de_r$ will lead to even $\Si$ and half to odd $\Si$.

Is the case $\ga_1\equiv\cdots\equiv\ga_{n-1}\equiv1~(\mod2)$ possible?
Recall that $\ga_i=0$ for $i>k$ and $\ga_i\in\{0,m/2\}$ for $i\le k$,
i.e.\ $\ga_i$ is odd if and only if $i\in\{1,\dots,k\}$, $\ga_i=m/2$ and $m\equiv2~(\mod4)$.

In the case $\ve=0$ the situation where $\ga_1\equiv\cdots\equiv\ga_{n-1}\equiv1~(\mod2)$ is only possible
if $m\equiv2~(\mod4)$, $\ga_1=\dots=\ga_k=m/2$ and $n=k+1$.
Recall that $n\equiv g+1~(\mod2)$, comparing with $n=k+1$ we obtain $k\equiv g~(\mod2)$.
Recall that $1-g=\ga_1+\cdots+\ga_k$, comparing with $\ga_1=\cdots=\ga_k=m/2$ we obtain $1-g\equiv k\cdot m/2~(\mod m)$.
Using $k\equiv g~(\mod2)$ and $m/2\equiv1~(\mod2)$ we obtain $1-g\equiv g~(\mod2)$.
This contradiction shows that for $\ve=0$ the situation where $\ga_1\equiv\cdots\equiv\ga_{n-1}\equiv1~(\mod2)$ is impossible.

On the other hand in the case $\ve=1$ the situation where $\ga_1\equiv\cdots\equiv\ga_{n-1}\equiv1~(\mod2)$ 
is possible for $m\equiv2~(\mod4)$ and $\ga_1=\cdots=\ga_{k-1}=m/2$.
In this situation $\Si$ is even.
\item
In the case $\ve=1$ and $m=2~(\mod4)$,
for any of the $2^{k-1}-1$ ways to choose
$$(\ga_1,\dots,\ga_{k-1})\ne(m/2,\dots,m/2),$$
the number of ways to choose $\de_1,\dots,\de_{k-1}$
so that $\Si$ is even and odd respectively is $m^{k-1}/2$.
However for $(\ga_1,\dots,\ga_{k-1})=(m/2,\dots,m/2)$ any of the $m^{k-1}$ choices of $\de_1,\dots,\de_{k-1}$
gives even $\Si$.
Therefore the number of choices of $\ga_i$ and $\de_i$ that give odd $\Si$ is
$$(2^{k-1}-1)\cdot\frac{m^{k-1}}2$$
and the number of choices of $\ga_i$ and $\de_i$ that give even $\Si$ is
$$(2^{k-1}-1)\cdot\frac{m^{k-1}}2+m^{k-1}=(2^{k-1}+1)\cdot\frac{m^{k-1}}{2}.$$
Thus the number of even and odd real $m$-Arf functions is
$$
  m^{2\tilde g}\cdot(2^{k-1}\pm1)\cdot\frac{m^{k-1}}2
  =m^{2\tilde g+k-1}\cdot\frac{2^{k-1}\pm1}2
  =m^g\cdot\frac{2^{k-1}\pm1}2
$$
respectively.

In all other cases, i.e.\ in the case $\ve=0$ and in the case $\ve=1$, $m\equiv0~(\mod4)$,
the numbers of even and odd real $m$-Arf functions are equal.
In the case~$k=0$ (which is only possible for $\ve=0$), there are $m^g$ different real $m$-Arf functions,
therefore the numbers of even and odd real $m$-Arf functions are both equal to
$$\frac{m^g}{2}.$$
In the case~$k\ge1$, there are $m^g\cdot 2^{k-1}$ different real $m$-Arf functions,
therefore the numbers of even and odd real $m$-Arf functions are both equal to
$$\frac{m^g\cdot 2^{k-1}}{2}=m^g\cdot 2^{k-2}.$$
\end{enumerate}
\end{proof}

Our main result for odd~$m$ is as follows:

\begin{thm}
\label{class-odd}
Let $m$ be odd.
\begin{enumerate}[1)]
\item
Let $\arf$ be a real $m$-Arf function on~$(P,\tau)$.
Then
\begin{align*}
  &\arf(a_i)=\arf(a_i')~\text{and}~\arf(b_i)=\arf(b_i')\quad\text{for}~i=1,\dots,\tilde g,\\
  &\arf(c_1)=\cdots=\arf(c_n)=0,\\
  &g=1~(\mod m).
\end{align*}
\item
Assume that
$$g=1~(\mod m).$$
Let a value set~$\calV$ in $(\z/m\z)^{4\tilde g+2n-2}$ be
$$
  (
   \al_1,\be_1,\dots,\al_{\tilde g},\be_{\tilde g},\al_1',\be_1',\dots,\al_{\tilde g}',\be_{\tilde g}',
   \ga_1,\dots,\ga_{n-1},\de_1,\dots,\de_{n-1}
  ).
$$
Assume that
\begin{align*}
  &\al_i=\al_i'~\text{and}~\be_i=\be_i'\quad\text{for}~i=1,\dots,\tilde g,\\
  &\ga_1=\cdots=\ga_{n-1}=0.
\end{align*}
Then there exists a real $m$-Arf function $\arf$ on $(P,\tau)$ with
$$\arf(a_i)=\al_1,~\arf(b_i)=\be_i,~\arf(a_i')=\al_i',~\arf(b_i')=\be_i',~\arf(c_i)=\ga_i,~\arf(d_i)=\de_i.$$
For this Arf function we have $\arf(c_n)=0$.
\item
The number of real $m$-Arf functions on~$(P,\tau)$ is $m^g$.
\end{enumerate}
\end{thm}

\begin{proof}
\begin{enumerate}[1)]
\item
For a real $m$-Arf function $\arf$ on~$(P,\tau)$
we have $\arf(c_i)=0$ for $i=1,\dots,k$ according to Lemma~\ref{lem-compatible-arf-zero-or-half-m}
and $\arf(c_i)=0$ for $i=k+1,\dots,n$ by Definition~\ref{def-real-arf}.
Lemma~\ref{lem-half-sum} implies
$\arf(c_1)+\cdots+\arf(c_n)=1-g~(\mod m)$ and $g=1~(\mod m)$.
\item
The condition $g=1~(\mod m)$ implies $2-2g=0~(\mod m)$,
hence, according to Proposition~\ref{prop-arf-compact},
the values~$\calV$ on~$\calB$ determine a unique $m$-Arf function~$\arf$ on~$P$.
According to Lemma~\ref{lem-checking-real}, to show that this $m$-Arf function~$\arf$ is real,
it is sufficient to show that $\arf(c_n)=0$.
Lemma~\ref{lem-half-sum} implies that
$$\arf(c_1)+\cdots+\arf(c_n)=1-g~(\mod m).$$
On the other hand $\ga_1=\cdots=\ga_{n-1}=0$, hence
$$\arf(c_1)+\cdots+\arf(c_n)=\ga_1+\cdots+\ga_{n-1}+\arf(c_n)=\arf(c_n)~(\mod m).$$
Comparing these equations we obtain
$$\arf(c_n)=1-g~(\mod m).$$
The condition $g=1~(\mod m)$ implies $\arf(c_n)=0~(\mod m)$.
Thus $\arf$ is a real $m$-Arf function.
\item
There are $m^{2\tilde g}$ way to choose $\al_i=\al_i'$ and $\be_i=\be_i'$,
only one way to choose $\ga_1,\dots,\ga_{n-1}$
and $m^{n-1}$ ways to choose $\de_1,\dots,\de_{n-1}$,
hence there are
$$m^{2\tilde g+n-1}=m^g$$
different choices of~$\calV$ and hence different real $m$-Arf functions.
\end{enumerate}
\end{proof}



\def\cprime{$'$}
\providecommand{\bysame}{\leavevmode\hbox to3em{\hrulefill}\thinspace}
\providecommand{\MR}{\relax\ifhmode\unskip\space\fi MR }
\providecommand{\MRhref}[2]{%
  \href{http://www.ams.org/mathscinet-getitem?mr=#1}{#2}
}
\providecommand{\href}[2]{#2}

\end{document}